\documentclass{amsproc}

 \newtheorem{teor}{Theorem}[section]

 \newtheorem{prop}[teor]{Proposition}
 \theoremstyle{definition}
 \newtheorem{defin}[teor]{Definition}
 \theoremstyle{remark}

\numberwithin{equation}{section}

\usepackage[all]{xy}
\usepackage{subfigure}

\usepackage{color}
\usepackage{graphics}
\usepackage{graphicx}
\usepackage{epsfig}
\usepackage{amssymb}
\usepackage{verbatim}
\usepackage{subfig}
\usepackage{graphics,amsmath,amsfonts,amssymb} 
\usepackage{epsfig,verbatim,multicol}

\begin{document}

\title[Transcritical Bifurcations and Algebraic Aspects of Quadratic Multiparametric Fam.]{Transcritical Bifurcations and Algebraic Aspects of Quadratic Multiparametric Families}

  \author[J. Rodriguez-Contreras]{Jorge Rodr\'iguez Contreras}
 \address[J. Rodriguez Contreras]{Departamento de matem\'aticas y estad\'istica Universidad del Norte \& Facultad de Ciencias B\'asicas Universidad del Atl\'antico, Barranquilla - Colombia-ORCID ID: https://orcid.org/0000-0002-1953-5350}
 \email{jrodri@uninorte.edu.co  \& jorgelrodriguezc@mail.uniatlantico.edu.co}
 
 \author[A. Reyes Linero]{Alberto Reyes Linero}
 \address[A. Reyes Linero]{Facultad de Ciencias B\'asicas Universidad del Atl\'antico, Barranquilla - Colombia- ORCID ID: https://orcid.org/0000-0002-9024-5006}
 \email{areyeslinero@mail.uniatlantico.edu.co}
 
  \author[B. Blanco Montes]{Bladimir Blanco Montes}
 \address[B. Blanco Montes]{Facultad de Ciencias B\'asicas Universidad del Atl\'antico, Barranquilla - Colombia
}
 \email{bablanco@mail.uniatlantico.edu.co}

 \author[P. Acosta-Humanez]{Primitivo B. Acosta-Hum\'anez}
 \address[P. Acosta-Humanez]{Facultad de Ciencias B\'asicas y Biom\'edicas, Universidad Sim\'on Bol\'ivar, Barranquilla - Colombia \& Instituto Superior de Formaci\'on Docente Salom\'e Ure\~na, Recinto Emilio Prud'Homme, Santiago - Dominican Republic - ORCID ID: https://orcid.org/0000-0002-5627-4188}
 \email{primitivo.acosta-humanez@isfodosu.edu.do}
 
 \maketitle

 \begin{abstract}
This article reveals an analysis of the quadratic systems that hold multiparametric families therefore, in the first instance the quadratic systems are identified and classified in order to facilitate their study and then the stability of the critical points in the finite plane, its bifurcations, stable manifold and lastly, the stability of the critical points in the infinite plane, afterwards the phase portraits resulting from the analysis of these families are graphed. To properly perform this study it was necessary to use some results of the non-linear systems theory, for this reason vital definitions and theorems were included because of their importance during the study of the multiparametric families. Algebraic aspects are also included.


 \end{abstract}

\section{Introduction}
Systems of differential equations are known to express a  number of mathematical, physical and engineering situations. In particular, this article is based about the study of all quadratic multiparametric subfamilies associated with the next family:
Given the family with $ a,b,c,m,k\in\mathbb{R} $.

\begin{equation}
\left\{ \begin{array}{lcl}
\dot{x}&=&y\\
\dot{y}&=&\left(\alpha x^{m+k-1}+\beta x^{m-k-1}\right)y-\gamma x^{2m-2k-1}\label{ecuacionprincipal} \\
\end{array}\right.
\end{equation}
 
We can find antecedents of the algebraic and qualitative studies of this family in \cite{1,2,2a,3}. Another algebraic and dynamical studies can be found in \cite{c1,c2}. In general, we can see qualitative studies about planar systems in \cite{6} ,\cite{4} and  \cite{5}, furthermore antecedents of  applied bifurcations study in \cite{VG}. In the present work, we take Proposition 4.1, pag 12, in \cite{2,2a}  which the goal of analyze each quadratic subfamily equivalently to (\ref{ecuacionprincipal}). Considering the constants $a$,$b$, $c$, and  $s,p,r\in\mathbb{Z}^{+}$. Then, we analyze different cases to determine quadratic systems attached to (\ref{ecuacionprincipal}) taking into account the regions in the space determined by the for the different parameters.\\

For the study of the quadratic multiparametric families described by (\ref{ecuacionprincipal}) where we use different topics studied in \cite{7},\cite{8},\cite{9} and \cite{10}. Then, we find the critical points associated with each quadratic family and analyzing their stability in both the finite and infinite planes, also we present a deeper study study determined by regions to see the changes in stability of the critical points and from here analyze bifurcations presented in some families. Finally, we used software like \cite{10} and \cite{11} for the detailed construction of the behaviors by means of the global phase portrait associated with each quadratic multiparametric family.

\section{Preliminaries}
In this section we provide the necessary theoretical background to understand the rest of the paper.\\

A planar polynomial system of degree $n$ is given by:
	\begin{equation}
		\begin{aligned}
		\dot{x}=\textit{P(x,y)} \\
		\dot{y}=\textit{Q(x,y)} \label{1}
		\end{aligned}
	\end{equation}		 
Where $\textit{P},\textit{Q}\in\mathbb{C}[x,y]$, and $n$ is given by $n=\max(\deg \textit{P},\deg \textit{Q})$

We denote the polynomial vector field associated to the system (\ref{1}) like: By $X:=(\textit{P},\textit{Q})$. The planar polynomial vector field $X$ can be also written in the form:
	\begin{center}
		$X=\textit{P}(x,y)\dfrac{\partial}{\partial x} + \textit{Q}(x,y)\dfrac{\partial}{\partial y}.$
	\end{center} 

A differential equations associate to polynomial vector field of the form (\ref{1}) is given by: 
	\begin{center}
		$\dfrac{dy}{dy}=\dfrac{\textit{Q(x,y)}}{\textit{P(x,y)}}.$
	\end{center}
	

\begin{teor} \textbf{Hyperbolic Singular Points Theorem}.\label{3} \\
\\
Let $(0,0)$ be an isolated singular point of the vector field $X$, given by,
\begin{equation}
\left\{
\begin{array}{lcl}
\dot{x}&=& ax+by+A(x,y) \\
\dot{y}&=& cx+dy+B(x,y)
\end{array}\right.
\end{equation}
Where $A$ and $B$ are analytic in a neighborhood of the origin with $A(0,0)=B(0,0)=DA(0,0)=DB(0,0)=0$. Let $\lambda_{1}$ and $\lambda_{2}$ be the eigenvalues of the linear part $DX(0,0)$ of the system at the origin. Then:\\
\begin{itemize}
	\item[(a)] If $\lambda_{1}$,$\lambda_{2}$ are real and $\lambda_{1}$$\lambda_{2}$ $< 0$, then $(0,0)$ is a saddle, where separatrix call $(0,0)$ in the directions given by the eigenvectors associated with $\lambda_{1}$ and $\lambda_{2}$.\\
	\item[(b)] If $\lambda_{1}$,$\lambda_{2}$ are real and $\lambda_{1}$$\lambda_{2}$ $>$ $0$, then $(0,0)$ is a
	node. If $\lambda_{1}$ $>0$($\lambda_{1}$$<$ $0$) then it is repelling or unstable (respectively attracting or stable).\\
	\item[(c)] If $\lambda_{1}=\alpha + \beta_{i}$ y $\lambda_{2}= \alpha - \beta_{i}$ with $\alpha,\beta \neq 0 $ then $(0,0)$ is a focus. If $\alpha>0$ or $(\alpha<0)$ it is repelling or unstable
	(respectively attracting or stable).\\
	\item[(d)] If $\lambda_{1}$ = $\beta_{i}$ and $\lambda_{2}$ =-$\beta_{i}$, then $(0,0)$ is a linear center, focus or a center.
\end{itemize}
\end{teor}

 for a more detailed study, please see	\cite[pág 71]{9}

\begin{teor}\rm \textbf{Non-Hyperbolic Singular Points Theorem}\\

Let $(0,0)$ be an isolated singular point of the vector field X given by: \label{5}
	\begin{equation}
	\left\{ \begin{array}{lcc}
	\dot{x}= y + A(x,y) \\
	\dot{y}= B(x,y) \\
	\end{array}\right.
	\end{equation}
Where $X$ and $Y$ are analytic in a neighborhood of the point $(0,0)$ y and considers the series expansion have expansions starting with terms of the second degree in $x$ and $y$. Let $y=f(x)=a_{2}x^{2}+a_{3}x^{3}+\ldots$ be the solution of the equation $y+A(x,y)=0$ in a neighborhood of the point $(0,0)$, and suppose you have the following series expansion of the function $F(x)=B(x,f(x))=ax^{m}(1+\ldots)$ y $G(x)=(\frac{\partial A}{\partial x} + \frac{\partial B}{\partial y})(x,f(x))=bx^{n}(1+\ldots)$ where $a\neq0$, $m\geq2$, and $n\geq1$. Then:\\
\begin{itemize}
	\item[(1)] If $G(x)\equiv0$ and $F(x)=ax^{m}\ldots$ for $m\in\mathbb{N}$ with $m\geq1$ and $a\neq0$, then:
	\subitem(\textit{i}) if $m$ is odd and $a>0$, then the origin of $X$ is a saddle and If $a<0$, then it is a center or a focus.
	\subitem(\textit{ii}) If $m$ is even then the origin of $X$ is a cusp.
	\item[(2)] If $F(x)=ax^{m}+\ldots$ and  $G(x)=bx^{n}+\ldots$ with $m\in\mathbb{N}$, $m\geq2$,
	$n\in\mathbb{N}$, $n\geq1$, $a\neq0$ and $b\neq0$. Then we have:
	\subitem(\textit{i}) If $m$ is even, and
	\subitem(\textit{i.a}) $m<2n+1$, then the origin of $X$ is a cusp.
	\subitem(\textit{i.b}) $m>2n+1$, then the origin of $X$ is a saddle-node.
	\subitem(\textit{ii}) If $m$ is odd and $a>0$, then the origin of $X$ is a saddle.
	\subitem(\textit{iii}) If $m$ is odd, $a<0$ and
	\subitem(\textit{iii.a}) $m<2n + 1$, or $m=2n+1$ and $b^{2}+4a(n+1)<0$, then the origin of $X$ is a center or a focus. 
	\subitem(\textit{iii.b}) $n$ is odd and either $m>2n+1$, or $m=2n+1$ and $b^{2}+4a(n+1)\geq0$. Then the phase portrait of the origin of $X$ consists of one hyperbolic and one elliptic sector.
	\subitem(\textit{iii.c}) $n$ is even and either $m>2n+1$, or $m=2n+1$ and $b^{2}+4a(n+1)\geq0$. Then the origin of $X$ is a node. The node is attracting if $b<0$ and repelling if $b>0$.
\end{itemize} 
\end{teor}
For a more detailed study, see \cite[pág 116]{9}

\subsection{Bifurcations}
We consider the system, depend on a parameter $\lambda$:

	\begin{equation}
		\dot{x}=f(x,\lambda) \label{bif} 
	\end{equation}	
	
If the change in $\lambda$ that produces a qualitative or topological change in the behavior of the planar system (\ref{bif}), his is called a Bifurcations.
This  can be a local bifurcation occurs when the change in the parameter causes a change in the stability of an equilibrium point. Global bifurcations normally occur in larger invariant sets of the system.\\ 

\textbf{Codimension - One Bifurcations}
These Bifurcations require the variation of a single parameter to occur in the system, all have a normal form, that is, a topologically equivalent system, either local or global to the initial system.\\

\textbf{Transcritical Bifurcations:}
A transcritical Bifurcations in critical point exists for every value of the parameter $\lambda$  but they exchange their stability with another critical point after the \textquotedblleft collision\textquotedblright  between them.\\ 

\textbf{Saddle-focus-saddle Bifurcations:}
\begin{defin} We Will call a Bifurcations saddle-focus-saddle is when a parameter change it implies that two critical points, one saddle, collapse in a focus and later they recover its original stability.
\end{defin}

For a more detailed study, see \cite[pág 51]{4} and \cite[pág 314]{10}

\subsection{Infinite Singular Points} \label{infinito}

Consider $\mathbb{R}^{3}$ the sphere \\  $S^{2}=\left\lbrace (x_{1}, x_{2}, x_{3})\in\mathbb{R}^{3}; x^{2}_{1}+x^{2}_{2}+ x^{2}_{3}=1\right\rbrace$ and the plane $\pi=\left\lbrace (x_{1}, x_{2}, x_{3})\in\mathbb{R}^{3}; x_{3}=1\right\rbrace$, is tangent to $S_{2}$ in the point $(0,0,1)$. Let $r$ a line through the origin $(0,0,0)$ and a point $P$ of $\pi$, then $r$ intercept $S^{2}$ in two points $P_{+}$ y $P_{-}$, where the first is in the upper open hemisphere $H_{+}=\left\lbrace (x_{1},x_{2},x_{3})\in S^{2}; x_{3}>0\right\rbrace $ and the second is in the lower open hemisphere $H_{-}=\left\lbrace (x_{1},x_{2},x_{3})\in S^{2}; x_{3}<0\right\rbrace$.\\
\\
The expression for $p(X)$ in local chart $(U_{1},\phi_{1})$ is given by:
\begin{equation}
\left\{ \begin{array}{lcl}
\dot{u}&=&v^{d}\left[ -uP(\frac{1}{v},\frac{u}{v}) + Q(\frac{1}{v},\frac{u}{v}) \right],\\
\dot{v}&=&-v^{d+1}P(\frac{1}{v},\frac{u}{v}).
\end{array}\right. \label{inf1}
\end{equation}
The expression for $(U_{2},\phi_{2})$ is:
\begin{equation}
\left\{ \begin{array}{lcl}
\dot{u}&=&v^{d}\left[ P(\frac{u}{v},\frac{1}{v}) - uQ(\frac{u}{v},\frac{1}{v}) \right],\\
\dot{v}&=&-v^{d+1}Q(\frac{u}{v},\frac{1}{v}).
\end{array}\right. \label{inf2}
\end{equation}
and for $(U_{3},\phi_{3})$ is:
\begin{equation}
\left\{ \begin{array}{lcl}
\dot{u}&=&P(u,v),\\
\dot{v}&=&Q(u,v). \label{inf3}
\end{array}\right.
\end{equation}
Where $d$ is the maximum degree of the polynomial.
For a more detailed study, see \cite[pág 151]{7}\\

\subsection{Algebraic Methods}
Concerning algebraic aspects considered in this paper, we follow the references \cite{12,13,14,15,16,17,18,19,20,21,22,23,24,25,26} and also \cite{1,2,2a,3,c1,c2}. 

A differential field $K$ is a field equiped with a derivation $\partial$ such that $\forall a,b\in K$ it satisfied:
\begin{enumerate}
    \item $\partial(a+b)=\partial a+\partial b$
    \item $\partial (a\cdot b)=a\cdot \partial b+\partial a\cdot b$
    \item $\partial \left(\frac{a}{b}\right)=\frac{1}{b^2}(a\cdot \partial b-\partial a\cdot b)$
\end{enumerate}
The field of constants of $K$, denoted by $C_K$, is given by $$C_K=\{c\in K:\,\, \partial (c)=0\}$$
The Picard-Vessiot extension $L/K$ is the extension of $K$ preserving the field of constants, that is $C_L=C_K$. Thus, given a system of first order linear differential equations $\dot X=AX$, where $a_{ij}\in K$, the differential Galois group of $\dot X=AX$, denoted by $DGal(L/K)$, is the group of $K$-differential automorphisms from $L$ to $L$, i.e., $\sigma:\,L\mapsto L$, $\partial(\sigma (a))=\sigma(\partial a)$,  $$DGal(L/K)=\{\sigma:\,\, \sigma(a)=a,\,  \forall a\in K\}$$
A Hamiltonian system of $n$ degrees of freedom $$\dot q_i=\frac{\partial H}{\partial p_i},\,\, \dot p_i=-\frac{\partial H}{\partial q_1}, \, 1\leq i\leq n,$$ where $\mathbf{q}=(q_1,\ldots,q_n)$, $\mathbf{p}=(p_1,\ldots,p_n)$ and $$H=\frac{\mathbf{p}\cdot \mathbf{p}}{2m}+V(\mathbf{q}),\,\, (\mathbf{q},\mathbf{p})\in \mathbb{R}^{2n},$$ is integrable in the Liouville sense whether there exist $n$ independent first integrals that commute in pairwise with the Poisson bracket. In particular, the Hamiltonian systems with one degree of freedom are integrable in the Liouville sense because $H$ is the first integral, i.e., $\dot H=0$. Morales-Ramis theory is the theory that relates differential Galois theory with the integrability of dynamical systems. In particular, Morales-Ramis Theorem for Hamiltonian systems says that \emph{if a Hamiltonian system is integrable, then the connected identity component of the differential Galois group of the first variational equation is an abelian group}.\\

On the other hand, explicit solutions for differential equation
\begin{equation}\label{lmp1}
    \frac{d^2x}{dt^2}=f(x)
\end{equation}
are related with the integral curve $(x,\dot x)$ of the one degree of freedom Hamiltonian system
\begin{equation}\label{lmp2}
\dot x=y,\,\, \dot y=f(x),\,\, H=\frac{y^2}{2}-\int_{x_0}^xf(\tau)d\tau
\end{equation}
We are interested in the families in where $f(x)$ is a polynomial of degree two, that is, family I, Family IV when $p=-4$ and Family V when $s=-4$. We recall that our problem comes from a polynomial vector field provided in \cite{2,2a}, for this reason $p,s\in \mathbb{Z}_0$ to get polynomial vector fields, while $p=-4$ and $s=-4$ correspond originally to a rational non-polynomial vector field, which is exceptionally considered for the algebraic aspects. 

Following \cite{27}, the Weierstrass $P$-function is an elliptic function that satisfy the elliptic curve
\begin{equation}\label{pwei0}
    y^2=4x^3-g_2x-g_3,\,\, x=\wp (t;g_2,g_3),\,\, y=\dot x.
\end{equation}
where
\begin{equation}\label{pwei1}
\wp (t;g_2,g_3)=\frac{1}{z^2}+\sum_\omega \left(\frac{1}{(t-\omega)^2}-\frac{1}{\omega^2}\right).    
\end{equation}
Morever, the Weierstrass $P$-function is a double periodic function with periods $2\omega_1$ and $2\omega_2$; and invariants $g_2$ and $g_3$ given by $$g_2=\sum_\omega \frac{60}{\omega^4},\quad g_3=\sum_\omega \frac{140}{\omega^6}.$$ The sums range over $\omega = 2n_1\omega_1 + 2n_2\omega_2$ such that $(n_1,n_2)\in\mathbb{Z}\times \mathbb{Z}\setminus \{(0,0)\}$.

\section{Conditions For The Problem.}

The following section allows us to identify the quadratic cases associated to (\ref{ecuacionprincipal}).

\subsection{Reduction to 5 Families}
The next proposition is a particular case of proposition 4.1 in \cite{2a}, we inly consider the quadratic cases.
\begin{prop} Let $a,b,c,m,k\in\mathbb{R}$ y $s,p,r\in\mathbb{Z}^{+}$. Quadratic systems associated with each subfamily of (\ref{ecuacionprincipal}) are equivalently to the following families:

		\begin{equation}{\textbf{I:}}
		\left\{
		\begin{array}{lcl}
		\dot{x}&=& y  \\
		\dot{y}&=& -cx^{2} \label{familia1}
		\end{array}\right.
		\end{equation}
		
		\begin{equation}{\textbf{II:}}
		\left\{
		\begin{array}{lcl}
		\dot{x}&=& y  \\
		\dot{y}&=& 2byx \label{familia2}
		\end{array}\right.
		\end{equation}
		
		\begin{equation}{\textbf{III:}}
		\left\{
		\begin{array}{lcl}
		\dot{x}&=& y  \\
		\dot{y}&=& 2ayx \label{familia3}
		\end{array}\right.
		\end{equation}
		
		\begin{equation}{\textbf{IV:}}
		\left\{
		\begin{array}{lcl}
		\dot{x}&=& y  \\
		\dot{y}&=& a\left( \frac{p+4}{2}\right)y-\frac{3}{2}a^{2}x-cx^{2} \label{familia4}
		\end{array}\right.
		\end{equation}
	
		\begin{equation}{\textbf{V:}}
		\left\{
		\begin{array}{lcl}
		\dot{x}&=& y  \\
		\dot{y}&=& b\left( \frac{s+4}{2}\right)y-\frac{3}{2}bx-cx^{2} \label{familia5}
		\end{array}\right.
		\end{equation}

\end{prop}
\begin{proof} We analyze each subfamily of the system (\ref{ecuacionprincipal}), where We observe the different possibilities for the constants $a,b$ and $c$, are equal to $0$. Some of cases are:
\begin{itemize}
\item[\textbf{I.}] For $a=0$ , $b=0$ and $c\neq0$. We observed two cases:\\
\textbf{Case 1.} If $s=0$, then $p=1$. ~ \textbf{Case 2.} If $s=1$, then $p=0$.\\
\item[\textbf{II.}] For $a\neq0$ , $b\neq0$ and $c\neq0$.\\

We observed that $deg(Q)=max\left\lbrace 2p+1,2s+1,s+p+1\right\rbrace$. \\
\textbf{Case 1.} If $deg(Q)=2p+1$, then $2p+1=2$ so $ p=\frac{1}{2}\notin\mathbb{Z}^{+}$.\\
\textbf{Case 2.} If $deg(Q)=2s+1$, then $2s+1=2$ so $ s=\frac{1}{2}\notin\mathbb{Z}^{+}$.\\
\textbf{Case 3.} If $deg(Q)=s+p+1$, then we return to reasoning in the family \textbf{I}, so $s=0$ then $p=1$, but we have that $2p+1=3$ and this case would be cubic. Same for $p=0$ and $s=1$.Therefore, this family does not have quadratic cases.
\end{itemize}
\end{proof}
\section{Finite Plane}\rm
\subsection{Singularity of the Family I }\rm 
\begin{prop} $(0,0)$ is a cusp of family (\ref{familia1}) .
\end{prop}
\begin{proof} The critical point associated with the system (\ref{familia1}) is $(0,0)$.
Jacobian matrix is:
\begin{center}
	$\mathcal{M}(x,y)=\left[
	\begin{array}{cl}
	0 & 1 \\
	-2cx & 0
	\end{array}\right]$
\end{center}
Then,
\begin{center}
	$\mathcal{M}(0,0)=\left[
	\begin{array}{cl}
	0 & 1 \\
	0 & 0
	\end{array}\right]$
\end{center}
We see that $\lambda^{2}=0$. Then, according to the Theorem (\ref{5}), where $A(x,y)=0$ and $y=0$, on the other hand We have to $B(x,y)=-cx^{2}$,we get that $F(x)=-cx^{2}$ and $G(x)=0$. Therefore the origin of the system (\ref{familia1}) is a cusp.\\
\end{proof}

\subsection{Singularity of the Family II}\rm 
\begin{prop}\rm The system \eqref{familia2} have infinite critical points. 
\end{prop}
\begin{proof}\rm $(x,0)$ which is a line of critical points associated with the system (\ref{familia2}). We see the solution of the system by separation of variables.
	\begin{center}
		$y=bx^{2}+k$, where $k$ is a constant.
	\end{center}
\end{proof}

\subsection{Singularity of the Family III }\rm 
\begin{prop}\rm The system \eqref{familia3} have infinite critical points. 
\end{prop}
\begin{proof}\rm $(x,0)$ which is a line of critical points associated with the system (\ref{familia3}). We see the solution of the system by separation of variables.
	\begin{center}
		$y=ax^{2}+k$, where $k$ is a constant.
	\end{center}
\end{proof}

\subsection{Singularity of the Family IV}\rm
\begin{prop}\rm 
	\begin{itemize}\rm
		\item[\textbf{a)}] The point $(0,0)$ is an stable node if $a<0$ and unstable if $a>0$, and $(\frac{-3a^{2}}{2c},0)$ is a saddle. 
		\item[\textbf{b)}] If $p=0$, $(0,0)$ is an stable focus if $a<0$ and unstable if $a>0$, and $(\frac{-3a^{2}}{2c},0)$ is a saddle.
	\end{itemize}
\end{prop}
\begin{proof}\rm Critical points associated with the system (\ref{familia4}) are: $(0,0)$ and $(\frac{-3a^{2}}{2c},0)$.\\
Let $d=a(p+4)$ and Jacobian matrix are:
	\begin{center}
		$\mathcal{M}(x,y)=\left[
		\begin{array}{cl}
		0  & 1 \\
		-\dfrac{3}{2}a^{2}-2cx & \dfrac{d}{2}
		\end{array}\right]$
	\end{center}
	Then,
	\begin{center}
		$\mathcal{M}(0,0)=\left[
		\begin{array}{cl}
		0 & 1 \\
		-\dfrac{3}{2}a^{2} & \dfrac{d}{2}
		\end{array}\right]$
	\end{center}
and,
\begin{center}
	$\mathcal{M}(\frac{-3a^{2}}{2c},0)=\left[
	\begin{array}{cl}
	0 & 1 \\
	\dfrac{3a^{2}}{2} & \dfrac{d}{2}
	\end{array}\right]$
\end{center} 
\begin{itemize}
	\item[\textbf{a.}] For $\mathcal{M}(0,0)$, eigenvalues are:\\
	$\lambda_{1}=\dfrac{1}{4}\left[ d + \sqrt{d^{2}-24a^2} \right] $ and $\lambda_{2}=\dfrac{1}{4}\left[ d - \sqrt{d^{2}-24a^2} \right]$. According to the Theorem (\ref{3}) we see that $\lambda_{1}\lambda_{2}>0$ therefore $(0,0)$ is an stable node if $a<0$ and unstable if $a>0$.\\
	\\
	Now, for $\mathcal{M}(\frac{-3a^{2}}{2c},0)$, eigenvalues are:\\
	$\lambda_{1}=\dfrac{1}{4}\left[ + \sqrt{d^{2}+24a^2} \right] $ y $\lambda_{1}=\dfrac{1}{4}\left[ d - \sqrt{d^{2}+24a^2} \right]$. According to the Theorem (\ref{3}) we see that $\lambda_{1}\lambda_{2}<0$, then $(\frac{-3a^{2}}{2c},0)$ is a saddle.
	\item[\textbf{b.}] If $p=0$, for $\mathcal{M}(0,0)$ eigenvalues are:\\
	$\lambda_{1}=\dfrac{a}{2}(2+i\sqrt{2})$ and $\lambda_{2}=\dfrac{a}{2}(2-i\sqrt{2})$. According to the Theorem (\ref{3}) we see that $(0,0)$ is an stable focus if $a<0$ and unstable if $a>0$.\\
	\\
	Now, for $\mathcal{M}(\frac{-3a^{2}}{2c},0)$, eigenvalues are:\\
	$\lambda_{1}=\dfrac{a}{4}\left[ 4 + 2\sqrt{10} \right] $ and $\lambda_{1}=\dfrac{a}{4}\left[ 4 - 2\sqrt{10} \right]$. According to the Theorem (\ref{3}) we see that $\lambda_{1}\lambda_{2}<0$, then $(\frac{-3a^{2}}{2c},0)$ is a saddle.
\end{itemize}
\end{proof}

\subsection{Singularity of the Family V}\rm
Before looking at the following proposition, We define the following regions:\\
\begin{center}
$\begin{array}{l}
R_{1}=\left\lbrace (b,c,d)\in\mathbb{R}^3 |  d^2-24b>0 \right\rbrace\\
R_{2}=\left\lbrace (b,c,d)\in\mathbb{R}^3 |  d^2-24b=0 \right\rbrace\\ 
R_{3}=\left\lbrace (b,c,d)\in\mathbb{R}^3 |  d^2-24b<0, c>0 \right\rbrace\\
R_{4}=\left\lbrace (b,0,d)\in\mathbb{R}^3 |  d^2-24b>0 \right\rbrace\\
R_{5}=\left\lbrace (b,0,d)\in\mathbb{R}^3 |  d^2-24b=0 \right\rbrace\\
R_{6}=\left\lbrace (b,0,d)\in\mathbb{R}^3 |  d^2-24b<0 \right\rbrace\\
R_{7}=\left\lbrace (0,c,0)\in\mathbb{R}^3 |  c>0 \right\rbrace\\
R_{8}=\left\lbrace (b,c,d)\in\mathbb{R}^3 |  c<0 \right\rbrace\\
\end{array}$
\end{center}

We note that $\mathbb{R}^3=\bigcup_{i=1}^{8} R_i$. Now in $R_3$ and $R_4$ we consider the following subsets:\\
\begin{center}
$E_{1}=\left\lbrace (b,c,d)\in\mathbb{R}^3 |  d^2-24b<0, d>0, c>0 \right\rbrace$\\
$E_{2}=\left\lbrace (b,c,d)\in\mathbb{R}^3 |  d^2-24b<0, d<0, c>0 \right\rbrace$\\
$E_{3}=\left\lbrace (b,c,d)\in\mathbb{R}^3 |  d^2+24b<0, d>0, c>0 \right\rbrace$\\
$E_{4}=\left\lbrace (b,c,d)\in\mathbb{R}^3 |  d^2-24b<0, d<0, c>0 \right\rbrace$\\
$E_{5}=\left\lbrace (b,c,d)\in\mathbb{R}^3 |  d^2-24b<0, d>0, c<0 \right\rbrace$\\
$E_{6}=\left\lbrace (b,c,d)\in\mathbb{R}^3 |  d^2-24b<0, d<0, c<0 \right\rbrace$\\
$E_{7}=\left\lbrace (b,c,d)\in\mathbb{R}^3 |  d^2+24b<0, d>0, c<0 \right\rbrace$\\
$E_{8}=\left\lbrace (b,c,d)\in\mathbb{R}^3 |  d^2-24b<0, d<0, c<0 \right\rbrace$\\
$E_{9}=\left\lbrace (b,c,d)\in\mathbb{R}^3 | d^2-24b>0, d>0, c>0 \right\rbrace$\\
$E_{10}=\left\lbrace (b,c,d)\in\mathbb{R}^3 | d^2-24b>0, d<0, c>0 \right\rbrace$\\
$E_{11}=\left\lbrace (b,c,d)\in\mathbb{R}^3 | d^2-24b>0, d<0, c<0 \right\rbrace$\\
$E_{12}=\left\lbrace (b,c,d)\in\mathbb{R}^3 | d^2-24b>0, d>0 c<0 \right\rbrace$\\

\end{center}

\begin{prop}\rm Let the family (\ref{familia5}) with $(b,c,d)\in\mathbb{R}^3$, then:  \\
	\begin{itemize}\rm
		\item[\textbf{a)}] If $(b,c,d)\in R_{1}$ and $b>0$ then the point $(0,0)$ is unstable node and the point $(\frac{-3b}{2c},0)$ is a saddle. if $b<0$ then the point $(0,0)$ is a saddle and the point $(\frac{-3b}{2c},0)$ is stable node.\\
		\item[\textbf{b)}] If $(b,c,d)\in R_{2}$ and $b>0$, then the critical point $(0,0)$ is a unstable node and the critical point $(\frac{-3b}{2c},0)$ is saddle.\\
		\item[\textbf{c)}] If $(b,c,d)\in R_{3}$ and $b>0$ then the point $(0,0)$ is stable focus and the point $(\frac{-3b}{2c},0)$ is a saddle. If $b<0$ then point  $(0,0)$ is a unstable focus and the point $(\frac{-3b}{2c},0)$ is a unstable node. 
	\end{itemize}
\end{prop}
\begin{proof} Let $d=b(s+4)$, so critical points associated with the system (\ref{familia5}) are: $(0,0)$ and $(\frac{-3b}{2c},0)$.\\
then, Jacobian matrix are:
	\begin{center}
		$\mathcal{M}(x,y)=\left[
		\begin{array}{cl}
		0  & 1 \\
		-\dfrac{3}{2}b-2cx & \dfrac{d}{2}
		\end{array}\right]$
	\end{center}
	Then,
	\begin{center}
		$\mathcal{M}(0,0)=\left[
		\begin{array}{cl}
		0 & 1 \\
		-\dfrac{3b}{2} & \dfrac{d}{2}
		\end{array}\right]$
	\end{center}

For $\mathcal{M}(0,0)$, eigenvalues are:

\begin{center}
    $\lambda_{1}=\dfrac{1}{4}\left[ d + \sqrt{d^{2}-24b} \right] $ and $\lambda_{2}=\dfrac{1}{4}\left[ d - \sqrt{d^{2}-24b} \right]$.
\end{center}

Now, for $(\frac{-3b}{2c},0)$, we have that
	\begin{center}
		$\mathcal{M}(\frac{-3b}{2c},0)=\left[
		\begin{array}{cl}
		0 & 1 \\
		\dfrac{3b}{2} & \dfrac{d}{2}
		\end{array}\right]$
	\end{center}
	
With eigenvalues:

\begin{center}
    $\lambda_{1}=\dfrac{1}{4}\left[ d + \sqrt{d^{2}+24b} \right] $ and $\lambda_{2}=\dfrac{1}{4}\left[ d - \sqrt{d^{2}+24b} \right]$.
\end{center}

	\begin{itemize}
		\item[\textbf{a.}] If $(b,c,d)\in R_{1}$ that is $d^2-24b>0$, for $\mathcal{M}(0,0)$ We see that $\lambda_{1}\lambda_{2}=\frac{-3b}{2}$ and $\lambda_{1}>0$, then  according to the Theorem (\ref{3}) the critical point $(0,0)$ is unstable node if $b>0$ and a saddle if $b<0$.\\
		\\
		Now, for $\mathcal{M}(\frac{-3b}{2c},0)$,  we have that $d^2+24b>48b$, then:\\
		\begin{enumerate}
		    \item If $b>0$, According to the Theorem (\ref{3}) We see that $\lambda_{1}\lambda_{2}=\frac{-3b}{2}$, and $b>0$ then   $(\frac{-3b}{2c},0)$ is a saddle. If $b<0$, $\lambda_{2}<0$, then stable node.\\
		    \item If $b<0$ and $d^2+24b\in [ 48b,0 )$, then the critical point $(\frac{-3b}{2c},0)$ is stable focus.\\
		    \item If $b<0$ and $d^2+24b\geq 0$, then the critical point $(\frac{-3b}{2c},0)$ if stable node.\\
		\end{enumerate}
	
		\item[\textbf{b.}] If $(b,c,d)\in R_{2}$ that is $d^2-24b=0$ this leans to $b\geq0$ for $\mathcal{M}(0,0)$ $\lambda_{1}=\lambda_{2}=\frac{d}{4}$, if $b>0$ then the critical point is unstable node. We note that if $b=0$, then this corresponding to (\ref{familia1}). Now, for $\mathcal{M}(\frac{-3b}{2c},0)$, if $b>0$ that is $d^2-24b>0$, Furthermore $\lambda_{1}\lambda_{2}=\frac{-3b}{2}$ according to the Theorem (\ref{3}) the critical point $(\frac{-3b}{2c},0)$ is saddle.\\
		
		\item[\textbf{c.}] If $(b,c,d)\in R_{3}$, that is $d^2-24b<0$ then  $b>0$. For $\mathcal{M}(0,0)$ eigenvalues are: 	$\lambda_{1}=\dfrac{1}{4}(d+i\sqrt{24b-d^2})$ and $\lambda_{2}=\dfrac{1}{4}(d-i\sqrt{24b-d^2})$, then according to the Theorem (\ref{3}), $(0,0)$ is a focus unstable. Now, for $b>0$ We see that $d^2+24b>0$, that is for $\mathcal{M}(\frac{-3b}{2c},0)$, we have that $\lambda_{1,2}\in \mathbb{R}$. So, $\lambda_{1}\lambda_{2}=\frac{-3b}{2}$, so According to the Theorem (\ref{3}) we have that the critical point $(\frac{-3b}{2c},0)$ is a saddle.
	\end{itemize}
\end{proof}

\begin{prop}\rm Given the family (\ref{familia5}) with $c=0$, then: 
\begin{itemize}
    \item[\textbf{a)}] If $(b,0,d)\in R_{4}$ and $b>0$, then the critical point $(0,0)$ is a saddle. If $b<0$, then the critical point $(0,0)$ is a stable node.\\
    \item[\textbf{b)}] If $(b,0,d)\in R_{5}$ and $b>0$, then the critical point $(0,0)$ is a unstable node.\\
    \item[\textbf{c)}] If $(b,0,d)\in R_{6}$ and $b>0$, then the critical point $(0,0)$ is a unstable focus. If $b<0$, then the critical point $(0,0)$ is a stable focus.
\end{itemize}
\end{prop}

\begin{proof}\rm With $c=0$, the family $(\ref{familia5})$ have the form: \label{anexo}

\begin{equation}
		\left\{
		\begin{array}{lcl}
		\dot{x}&=& y  \\
		\dot{y}&=& \frac{d}{2}y-\frac{3}{2}bx \label{familiac=0}
		\end{array}\right.
\end{equation}

Here, we see that the only critical point associated with the family (\ref{familiac=0}) is $(0,0)$    
then, Jacobian matrix evaluated in the point is:
	\begin{center}
		$\mathcal{M}(0,0)=\left[
		\begin{array}{cl}
		0 & 1 \\
		-\dfrac{3b}{2} & \dfrac{d}{2}
		\end{array}\right]$
	\end{center}
For $\mathcal{M}(0,0)$, eigenvalues are:
\begin{center}
    $\lambda_{1}=\dfrac{1}{4}\left[ d + \sqrt{d^{2}-24b} \right] $ and $\lambda_{2}=\dfrac{1}{4}\left[ d - \sqrt{d^{2}-24b} \right]$.
\end{center}
		\begin{itemize}
		    \item [\textbf{a.}] If $(b,0,d)\in R_{4}$ that is $d^2-24b>0$, we have that $\lambda_{1}\lambda_{2}=\frac{-3b}{2}$, then if $b>0$ according to the Theorem (\ref{3}) the critical point $(0,0)$ is a saddle and if $b<0$ and $\lambda_{1}<0$ then the critical point $(0,0)$ is a stable node.\\
		    
		    \item [\textbf{b.}] If $(b,0,d)\in R_{5}$ that is $d^2-24b=0$, we have that $\lambda_{1}\lambda_{2}=\frac{d}{4}$, then according to the Theorem (\ref{3}) the critical point $(0,0)$ is a unstable node.\\
		    
		    \item [\textbf{c.}] If $(b,0,d)\in R_{6}$ that is $d^2-24b<0$, we have that $\lambda_{1}=\dfrac{1}{4}(d+i\sqrt{24b-d^2})$ and $\lambda_{2}=\dfrac{1}{4}(d-i\sqrt{24b-d^2})$, then if $b>0$ according to the Theorem (\ref{3}) the critical point $(0,0)$ is a unstable focus and if $b<0$ and $\lambda_{1}<0$ then the critical point $(0,0)$ is a stable focus.
		\end{itemize}
\end{proof}

Now, We will look for the stable manifold associated systems.
\begin{prop}\rm The stable manifold associated with the system (\ref{familia4}) at the point $(\frac{-3a^{2}}{2c},0)$ is: 
	\begin{center}
		$S:y=\frac{c(x+\frac{3a^{2}}{2c})^{2}}{(v-w)(v-2w)}$ 
	\end{center}
\end{prop}
\begin{proof}\rm Let is observe the stability of the system (\ref{familia4}) in the point $(\frac{-3a^{2}}{2c},0)$:\\
	Let is look at the eigenvalues in the Jacobian matrix of the system (\ref{familia4}) in the point $(\frac{-3a^{2}}{2c,0})$.
	\\
\begin{center}
		$A=\left[
	\begin{array}{cl}
	0 & 1 \\
	\dfrac{3a^{2}}{2} & \dfrac{d}{2}
	\end{array}\right]$ 
\end{center}
	So, $w=\lambda_{1}=\dfrac{1}{4}[d+\sqrt{d^{2}+24a^2}]$ \hspace{2mm} and \hspace{2mm} $v=\lambda_{2}=\dfrac{1}{4}[d-\sqrt{d^{2}+24a^2}].$\\
	\\
	Then, $ B(x)=C^{-1}AC=\left[
	\begin{array}{lcl}
	w & 0 \\
	0 & v
	\end{array}\right]$ \\
	\\
	\\
	$F(x)=\left[\begin{array}{c}
	0\\
	-cx^{2}
	\end{array} \right], \hspace{2mm} G(x)=\dfrac{cx^{2}}{v-w}\left[\begin{array}{cl}
	1\\
	-1
	\end{array} \right]$\\
	\begin{center}
		$U(t)=\left[\begin{array}{lcl}
		e^{wt} & 0\\[2mm]
		0 & 0
		\end{array} \right], \hspace{2mm} V(t)=\left[\begin{array}{lcl}
		0 & 0\\[2mm]
		0 & e^{vt}
		\end{array} \right],\hspace{2mm} b=\left[\begin{array}{c}
		b_{1}\\[2mm]
		0
		\end{array} \right]$
	\end{center}
	Then,\\
	\\
	$u^{(0)}(t,b)=0.$ \\
	\\
	$u^{(1)}(t,b)=\left[\begin{array}{c}
	e^{wt}b_{1}\\
	0
	\end{array} \right]$\\
	\\
	$u^{(2)}(t,b)=\left[\begin{array}{c}
	e^{wt}b_{1}\\
	0 
	\end{array} \right]+\int_{0}^{t}{\left[\begin{array}{lcl}
		e^{w(t-s)} & 0\\
		0 & 0.
		\end{array} \right]\left[\begin{array}{c}
		\frac{b_{1}^{2}c}{v-w}e^{2ws}\\
		-\frac{b_{1}^{2}c}{v-w}e^{2ws}
		\end{array} \right]}ds-$
	$\int_{t}^{\infty}{\left[\begin{array}{lcl}
		0 & 0\\
		0 & e^{v(t-s)}
		\end{array} \right]\left[\begin{array}{c}
		\frac{b_{1}^{2}c}{v-w}e^{2ws}\\
		-\frac{b_{1}^{2}c}{v-w}e^{2ws}
		\end{array} \right]}ds=\\
	u^{(2)}(t,b)=\left[\begin{array}{c}
	e^{wt}b_{1}+\frac{b_{1}^{2}ce^{2wt}[e^{wt}-1]}{w(v-w)}\\
	\\
	\frac{b_{1}^{2}ce^{2wt}}{(v-w)(v-2w)}
	\end{array} \right]
	$\\
	\\
	Therefore, We can approximate by $\psi_{2}(b_{1})=b_{1}$, therefore  the stable manifold can be approximated by 
	\begin{center}
		$S:y=\frac{c(x+\frac{3a^{2}}{2c})^{2}}{(v-w)(v-2w)}$ 
	\end{center}
	like $x \to 0$. Similarly the unstable
	\begin{center}
		$U: x + \frac{3a^{2}}{2c} = \frac{2cy^{2}}{(v-w)((v-2w)}$
	\end{center}
\end{proof}

\begin{prop}\rm For the system (\ref{familia5}) We have that:
	\begin{itemize}
		\item[\textbf{a)}] If $(b,c,d)\in R_{1}$ and $b<0$, stable manifold at the point $(0,0)$ is:
		\begin{center}
			$S:y=\frac{cx^{2}}{(v-w)(v-2w)}$ 
		\end{center}
		\item[\textbf{b)}] If $(b,c,d)\in \lbrace\left(x,y,z)/ x>0, y\neq0 \right\rbrace $, then stable manifold at the point  $(\frac{-3b}{2c},0)$ es: 
		\begin{center}
			$S:y=\frac{c(x+\frac{3b}{2c})^{2}}{(v-w)(v-2w)}$ 
		\end{center}
	\end{itemize}
\end{prop}

\begin{proof}\rm 
	\textbf{a)} Let is observe the stability of the system (\ref{familia5}) for $b<0$ at the point $(0,0)$:\\
Let is observe the stability of the system (\ref{familia5}) in the point $(0,0)$.
	\\
	\begin{center}
		$A=\left[
	\begin{array}{cl}
	0 & 1 \\
	\dfrac{-3b}{2} & \dfrac{d}{2}
	\end{array}\right]$ 
	\end{center}

	Let, $w=\lambda_{1}=\dfrac{1}{4}[d+\sqrt{d^{2}-24b}]$ \hspace{2mm} and \hspace{2mm} $v=\lambda_{2}=\dfrac{1}{4}[d-\sqrt{d^{2}-24b}].$\\
	\\
	Then, $ B(x)=C^{-1}AC=\left[
	\begin{array}{lcl}
	w & 0 \\
	0 & v
	\end{array}\right]$ \\
	\\
	\\
	$F(x)=\left[\begin{array}{c}
	0\\
	-cx^{2}
	\end{array} \right], \hspace{2mm} G(x)=\dfrac{cx^{2}}{v-w}\left[\begin{array}{cl}
	1\\
	-1
	\end{array} \right]$\\
	\begin{center}
		$U(t)=\left[\begin{array}{lcl}
		e^{wt} & 0\\[2mm]
		0 & 0
		\end{array} \right], \hspace{2mm} V(t)=\left[\begin{array}{lcl}
		0 & 0\\[2mm]
		0 & e^{vt}
		\end{array} \right],\hspace{2mm} a=\left[\begin{array}{c}
		a_{1}\\[2mm]
		0
		\end{array} \right]$
	\end{center}
	Then,\\
	\\
	$u^{(0)}(t,a)=0.$ \\
	\\
	$u^{(1)}(t,a)=\left[\begin{array}{c}
	e^{wt}a_{1}\\
	0
	\end{array} \right]$\\
	\\
	$u^{(2)}(t,a)=\left[\begin{array}{c}
	e^{wt}a_{1}\\
	0 
	\end{array} \right]+\int_{0}^{t}{\left[\begin{array}{lcl}
		e^{w(t-s)} & 0\\
		0 & 0.
		\end{array} \right]\left[\begin{array}{c}
		\frac{a_{1}^{2}c}{v-w}e^{2ws}\\
		-\frac{a_{1}^{2}c}{v-w}e^{2ws}
		\end{array} \right]}ds-$
	$\int_{t}^{\infty}{\left[\begin{array}{lcl}
		0 & 0\\
		0 & e^{v(t-s)}
		\end{array} \right]\left[\begin{array}{c}
		\frac{a_{1}^{2}c}{v-w}e^{2ws}\\
		-\frac{a_{1}^{2}c}{v-w}e^{2ws}
		\end{array} \right]}ds=\\
	u^{(2)}(t,a)=\left[\begin{array}{c}
	e^{wt}a_{1}+\frac{a_{1}^{2}ce^{2wt}[e^{wt}-1]}{w(v-w)}\\
	\\
	\frac{a_{1}^{2}ce^{2wt}}{(v-w)(v-2w)}
	\end{array} \right]
	$\\
	\\
	Therefore, We can approximate by $\psi_{2}(a_{1})=a_{1}$, therefore  the stable manifold can be approximated by 
	\begin{center}
		$S:y=\frac{cx^{2}}{(v-w)(v-2w)}$ 
	\end{center}
	Like $x \to 0$. Similarly the unstable
	\begin{center}
		$U: x= \frac{2cy^{2}}{(v-w)((v-2w)}$
	\end{center}
	\textbf{b)} Let is observe the stability of the system (\ref{familia5}) at the point $(\frac{-3b}{2c},0)$, when $b>0$:\\
	\\
	\begin{center}
		$A=\left[
	\begin{array}{cl}
	0 & 1 \\
	\dfrac{3b}{2} & \dfrac{d}{2}
	\end{array}\right]$
	\end{center}
	Let, $w=\lambda_{1}=\dfrac{1}{4}[d+\sqrt{d^{2}+24b}]$ \hspace{2mm} and \hspace{2mm} $v=\lambda_{2}=\dfrac{1}{4}[d-\sqrt{d^{2}+24b}].$\\
	\\
	So, $ B(x)=C^{-1}AC=\left[
	\begin{array}{lcl}
	w & 0 \\
	0 & v
	\end{array}\right]$ \\
	\\
	\\
	$F(x)=\left[\begin{array}{c}
	0\\
	-cx^{2}
	\end{array} \right], \hspace{2mm} G(x)=\dfrac{cx^{2}}{v-w}\left[\begin{array}{cl}
	1\\
	-1
	\end{array} \right]$\\
	\begin{center}
		$U(t)=\left[\begin{array}{lcl}
		e^{wt} & 0\\[2mm]
		0 & 0
		\end{array} \right], \hspace{2mm} V(t)=\left[\begin{array}{lcl}
		0 & 0\\[2mm]
		0 & e^{vt}
		\end{array} \right],\hspace{2mm} a=\left[\begin{array}{c}
		a_{1}\\[2mm]
		0
		\end{array} \right]$
	\end{center}
	Then,\\
	\\
	$u^{(0)}(t,a)=0.$ \\
	\\
	$u^{(1)}(t,a)=\left[\begin{array}{c}
	e^{wt}a_{1}\\
	0
	\end{array} \right]$\\
	\\
	$u^{(2)}(t,a)=\left[\begin{array}{c}
	e^{wt}a_{1}\\
	0 
	\end{array} \right]+\int_{0}^{t}{\left[\begin{array}{lcl}
		e^{w(t-s)} & 0\\
		0 & 0.
		\end{array} \right]\left[\begin{array}{c}
		\frac{a_{1}^{2}c}{v-w}e^{2ws}\\
		-\frac{a_{1}^{2}c}{v-w}e^{2ws}
		\end{array} \right]}ds-$
	$\int_{t}^{\infty}{\left[\begin{array}{lcl}
		0 & 0\\
		0 & e^{v(t-s)}
		\end{array} \right]\left[\begin{array}{c}
		\frac{a_{1}^{2}c}{v-w}e^{2ws}\\
		-\frac{a_{1}^{2}c}{v-w}e^{2ws}
		\end{array} \right]}ds=\\
	u^{(2)}(t,a)=\left[\begin{array}{c}
	e^{wt}a_{1}+\frac{a_{1}^{2}ce^{2wt}[e^{wt}-1]}{w(v-w)}\\
	\\
	\frac{a_{1}^{2}ce^{2wt}}{(v-w)(v-2w)}
	\end{array} \right]
	$\\
	\\
	Therefore, We can approximate by $\psi_{2}(a_{1})=a_{1}$, therefore  the stable manifold can be approximated by 
	\begin{center}
		$S:y=\frac{c(x+\frac{3b}{2c})^{2}}{(v-w)(v-2w)}$ 
	\end{center}
	Like $x \to 0$. Similarly the unstable
	\begin{center}
		$U: x + \frac{3b}{2c} = \frac{2cy^{2}}{(v-w)((v-2w)}$
	\end{center}
\end{proof}

\section{Bifurcations}
In this section we will analyze the study of the bifurcations of family (\ref{familia5})
	

\subsection{Family V}
\begin{prop} Let sets $R_{7}$ and $R_{8}$ are transcritical bifurcations for the system (\ref{familia5}) 
\end{prop}
\begin{proof}\rm Let $P_{1}:(0,0)$ and $P_{2}:(\frac{3b}{2c},0)$ of proposition $(4.5)$. If $(b,c,d)\in E_{3}$ then $P_{1}$ a saddle and $P_{2}$ is a unstable focus, when $(b,c,d)\in R_{7}$, $P_{1}$ and $P_{2}$, they collapse on one critical point which point is a cusp. So, when $(b,c,d)\in E_{2}$ then $P_{1}$ a unstable focus and $P_{2}$ is a saddle. Similarly, the same behavior is observed when $(b,c,d)\in E_{2}$, then $(b,c,d)\in R_{7}$ and finally $(b,c,d)\in E_{4}$.\\

Now,  Let $P_{1}:(0,0)$ and $P_{2}:(\frac{3b}{2c},0)$ of proposition $(4.5)$. If $(b,c,d)\in E_{4}$ then $P_{1}$ a saddle and $P_{2}$ is a stable focus, when $(b,c,d)\in R_{8}$, $P_{1}$ and $P_{2}$, they collapse on one critical point which point is a cusp. So, when $(b,c,d)\in E_{1}$ then $P_{1}$ a stable focus and $P_{2}$ is a saddle. Similarly, the same behavior is observed when $(b,c,d)\in E_{1}$, then $(b,c,d)\in R_{8}$ and finally $(b,c,d)\in E_{3}$.\\
Therefore, sets $R_{7}$ and $R_{8}$ are transcritical bifurcations for the system (\ref{familia5})
\end{proof}

\begin{prop}\rm A set $\{(b,0,d) | d^2-24b<0\}$ is a bifurcations saddle-focus-saddle for the system (\ref{familia5}).
\end{prop}

\begin{proof}\rm For proposition $(4.5)$, if $(b,c,d) \in E_{1}$, the point $P_{1}$ is a stable focus and $P_{2}$ is a saddle. Now, when $(b,c,d)\in \left\lbrace (b,0,d) | d^2-24b<0 , d>0 \right\rbrace $ for proposition (\ref{anexo}), $P_{1}$ and $P_{2}$ they collapse in an unstable focus when $(b,c,d)$ goes the set $E_{5}$, appear again $P_{1}$ and $P_{2}$ like stable focus and a saddle respectively.
\end{proof}

\begin{prop}\rm A set $\{(b,0,d) | d^2-24b<0\}$ is a bifurcations saddle-focus-saddle for the system (\ref{familia5}).
\end{prop}

\begin{proof}\rm For proposition $(4.5)$, if $(b,c,d) \in E_{2}$, the point $P_{1}$ is a unstable focus and $P_{2}$ is a saddle. Now, when $(b,c,d)\in \left\lbrace (b,0,d) | d^2-24b<0 , d<0 \right\rbrace $ for proposition (\ref{anexo}), $P_{1}$ and $P_{2}$ they collapse in an stable focus when $(b,c,d)$ goes the set $E_{6}$, appear again $P_{1}$ and $P_{2}$ like unstable focus and a saddle respectively.
    
\end{proof}

\begin{prop}\rm Let sets $E_{9}$ and $E_{10}$ are local bifurcations for the system (\ref{familia5})
\end{prop}

\begin{proof}\rm Let $P_{1}:(0,0)$ of proposition $(4.5)$. If $(b,c,d)\in E_{9}$ then $P_{1}$ is a stable node. Now, when $(b,c,d)\in E_{10}$ the point $P_{1}$ is a unstable node. Therefore, regions $E_{9}$ and $E_{10}$ are local bifurcations for the system (\ref{familia5})
\end{proof}

\begin{prop}\rm Let sets $E_{11}$ and $E_{12}$ are local bifurcations for the system (\ref{familia5})
\end{prop}

\begin{proof}\rm Let $P_{2}:(\frac{3b}{2c},0)$ of proposition $(4.5)$. If $(b,c,d)\in E_{12}$ then $P_{2}$ is a unstable node. Now, when $(b,c,d)\in E_{11}$ the point $P_{2}$ is a stable node. Therefore, regions $E_{11}$ and $E_{12}$ are local bifurcations for the system (\ref{familia5})
\end{proof}

\begin{figure}[htb]
	\centering
		\includegraphics[width=60mm]{"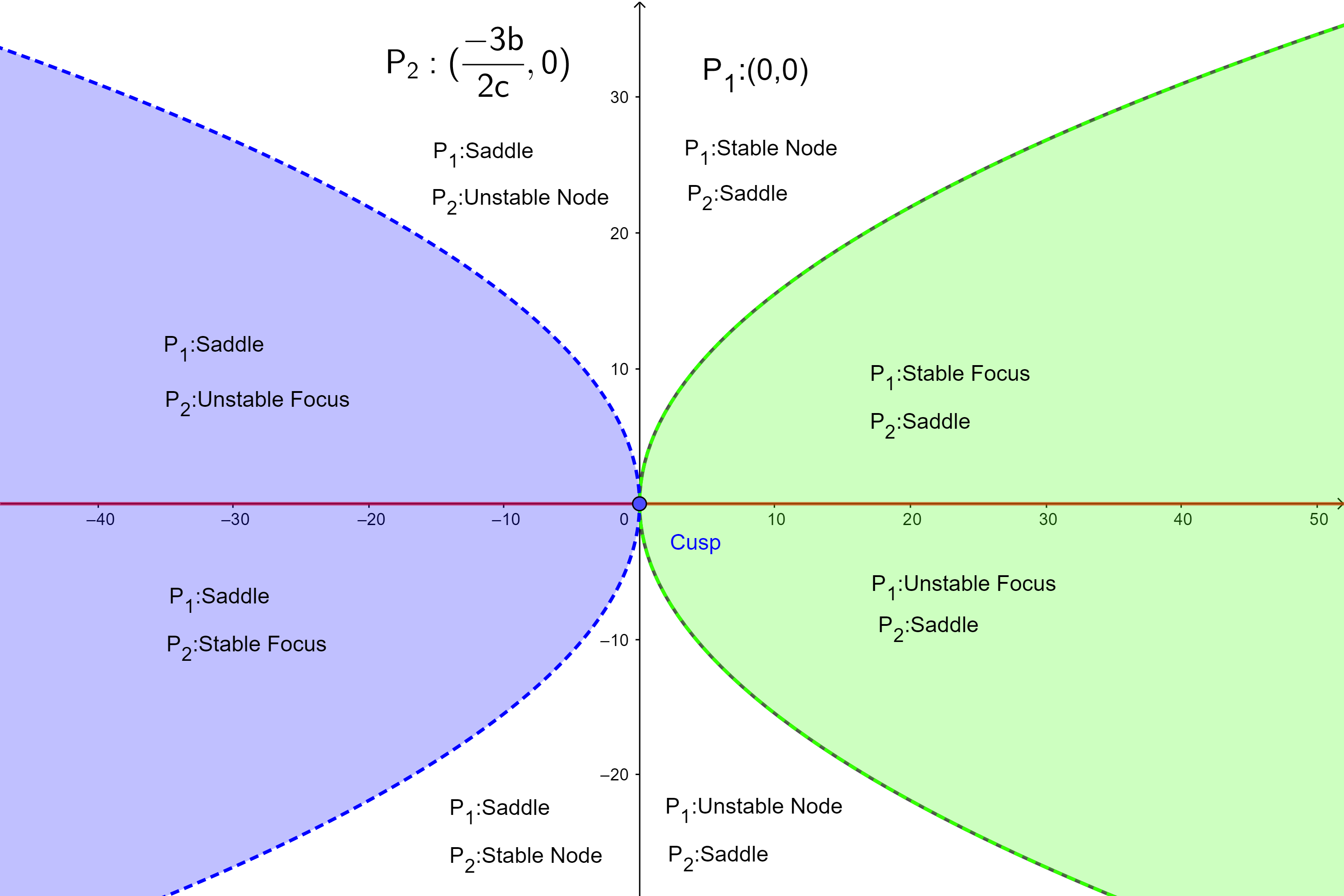"}
		\caption{\eqref{familia5}, $c>0$.}
		\label{fig:figura1}
 	\end{figure}
	
	\begin{figure}[htb]
	\centering
		\includegraphics[width=60mm]{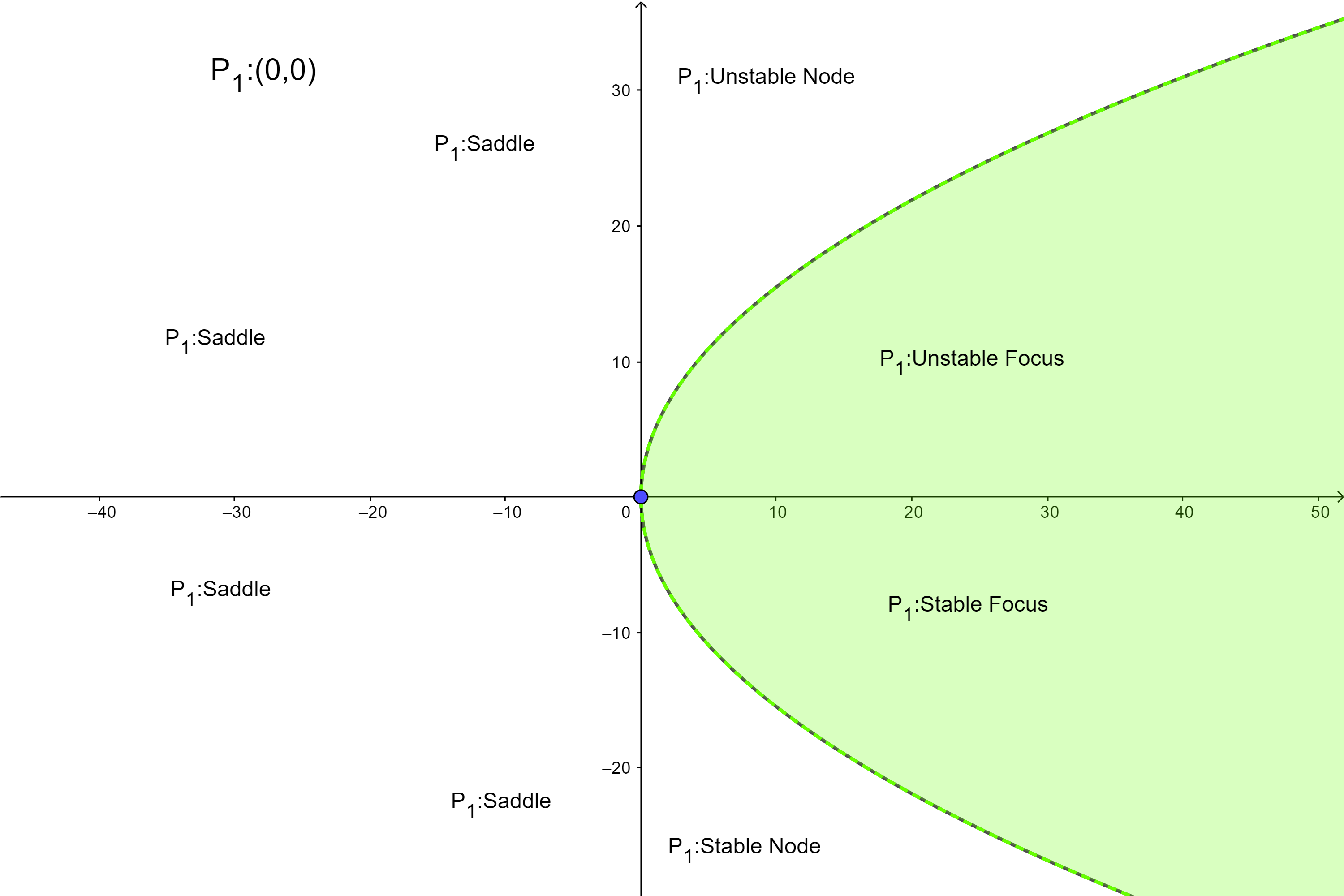}
		\caption{\eqref{familia5}, $c=0$.}
		\label{fig:figura2}

	\centering
		\includegraphics[width=60mm]{"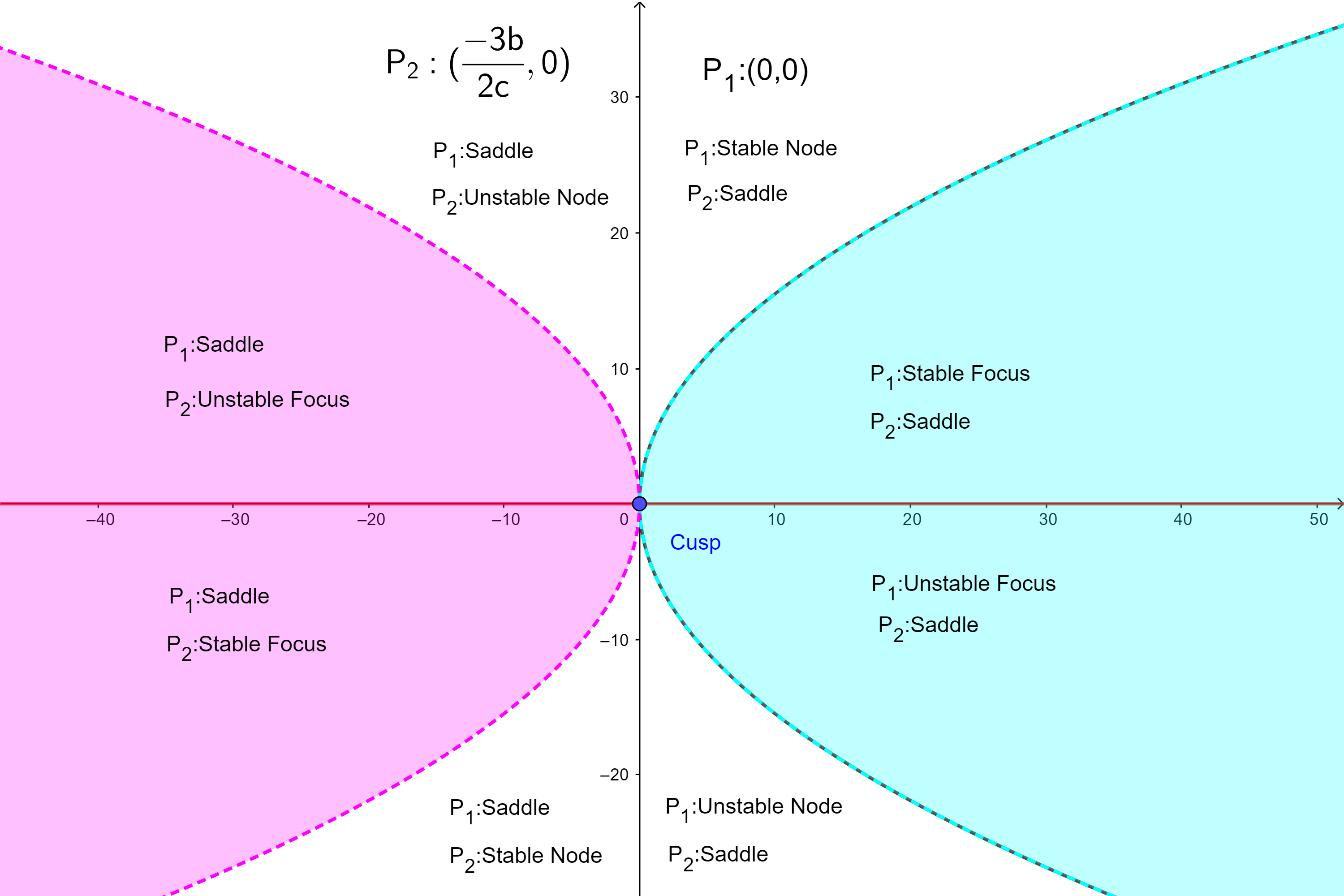"}
		\caption{\eqref{familia5}, $c<0$.}
		\label{fig:figura3}
	\end{figure}
	
	\begin{figure}[htb]
		\centering
		\includegraphics[width=60mm]{"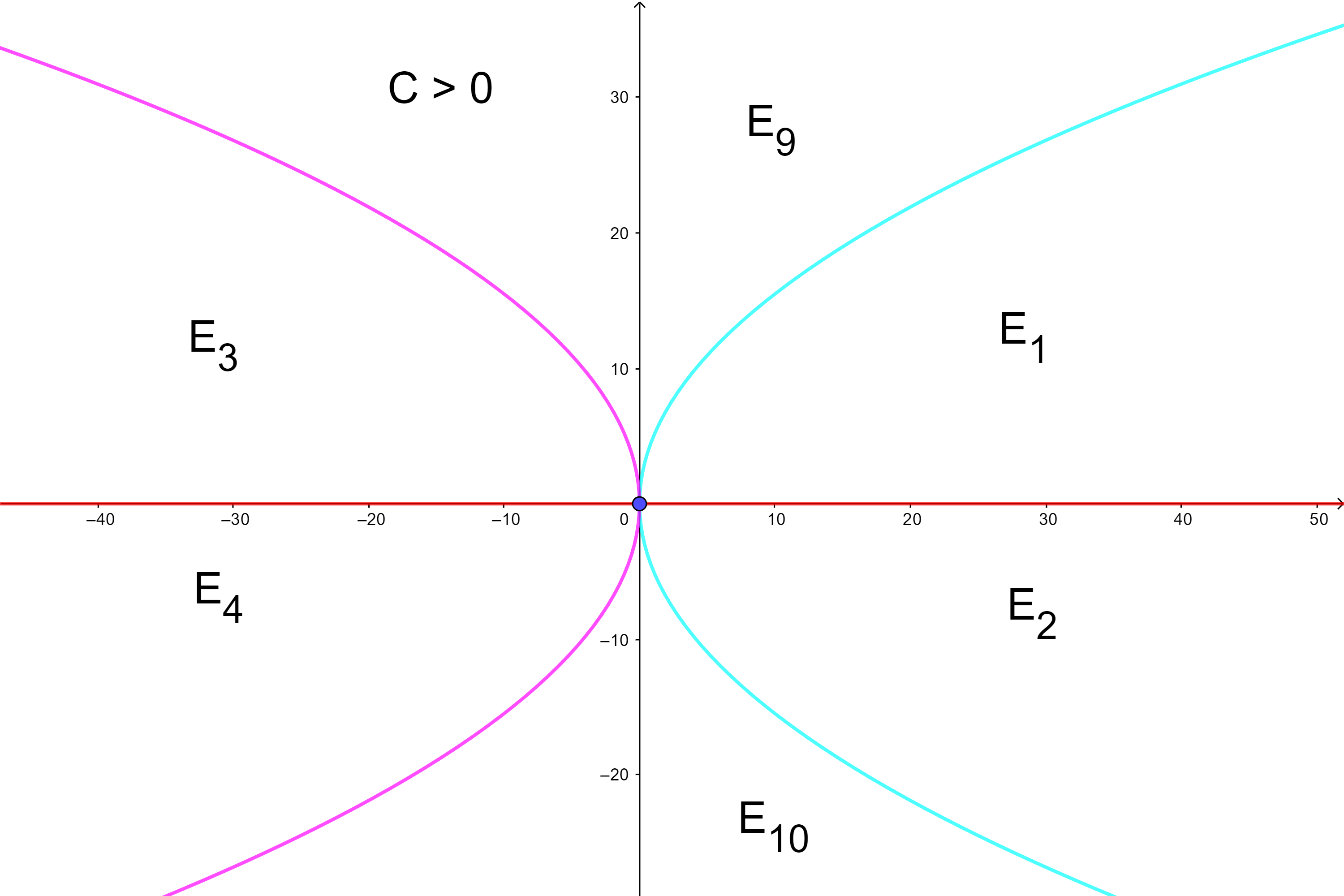"}
		\caption{\eqref{familia5}, $c>0$.}
		\label{fig:figura4}
	\end{figure}
	
	\begin{figure}[htb]
	\centering
		\includegraphics[width=60mm]{"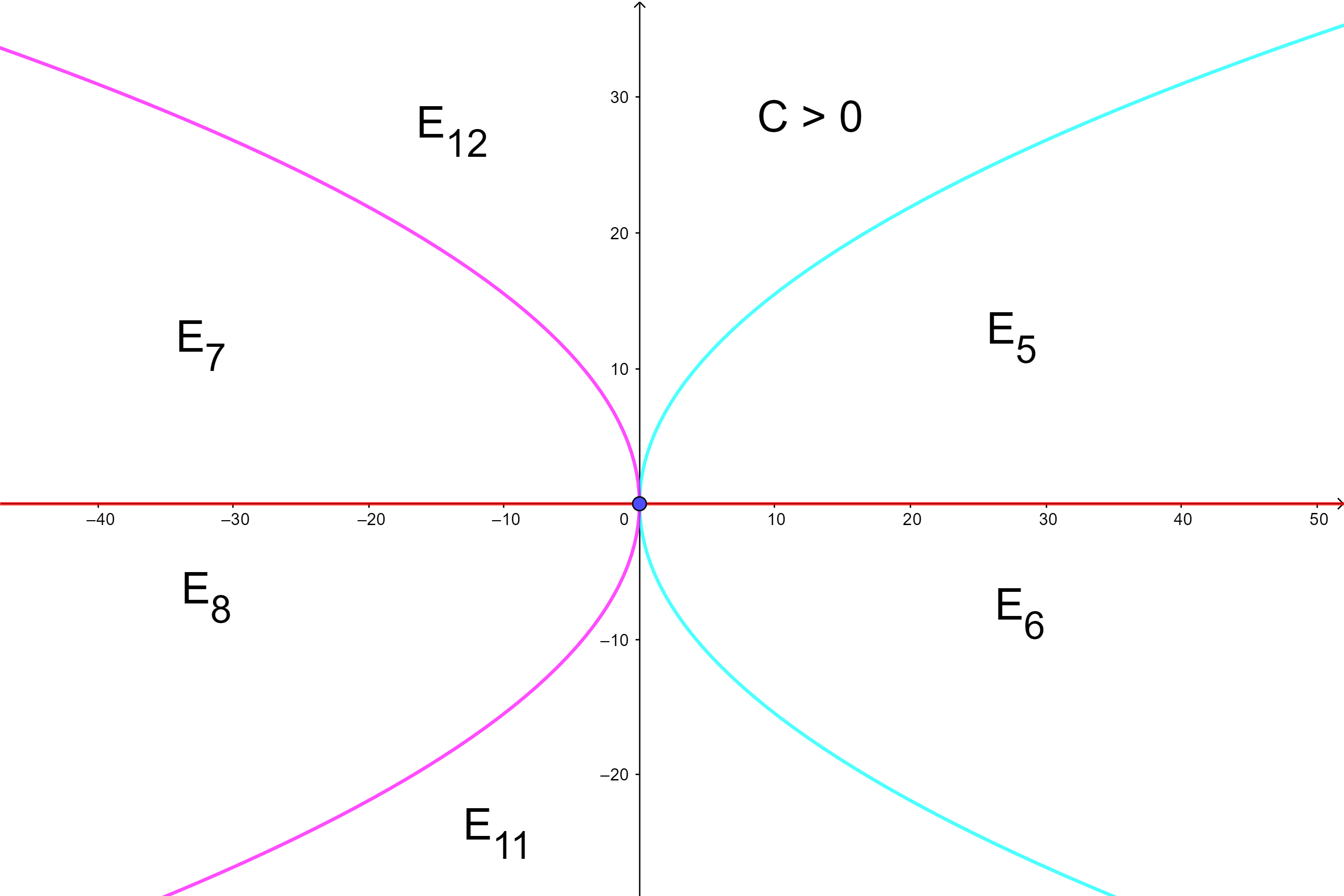"}
		\caption{\eqref{familia5}, $c<0$.}
		\label{fig:figura5}
	\end{figure}
	
	\begin{figure}[htb!]
	\centering
		\includegraphics[width=60mm]{"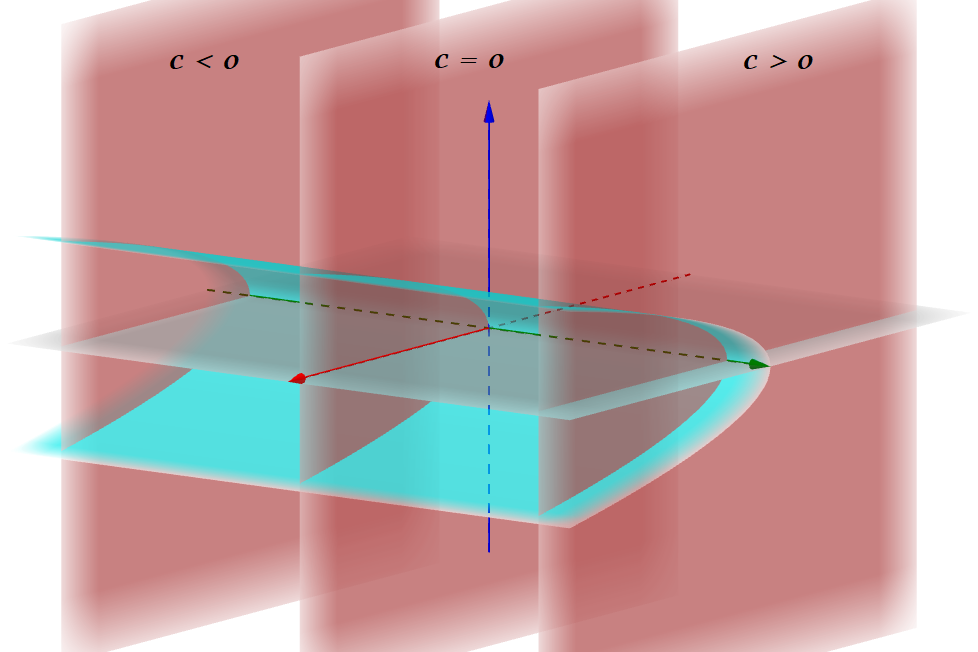"}
		\caption{\eqref{familia5}, $c<0$.}
		\label{fig:figura6}
	\end{figure}

\newpage
\section{Infinite Plane}
\subsection{Family I}\rm
\textbf{In the Chart $U_{1}$} the associated system \ref{familia1} : 
\begin{equation}\rm
	\left\{
	\begin{array}{lcl}
		\dot{u}&=& -u^{2}v-c  \\
		\dot{v}&=& -uv^{2} 
	\end{array}\right.
\end{equation}
The system have not critical points at infinite plane.\\
\\
\textbf{In the Chart $U_{2}$} the associated system \ref{familia1}:
\begin{equation}\rm
	\left\{
	\begin{array}{lcl}
		\dot{u}&=& v+cu^{3}  \\
		\dot{v}&=& -cu^{2}v \label{2.9}
	\end{array}\right.
\end{equation}
\begin{prop} The point $(0,0)$ is an stable node if $c<0$ and unstable if $c>0$.
\end{prop}
\begin{proof}\rm The critical points associated with the system (\ref{2.9}) is $P:(0,0)$.\\
	Jacobian matrix is:
	\begin{center}
		$\mathcal{M}(u,v)=\left[
		\begin{array}{lc}
		3cu^{2} & 1 \\
		-2cuv & -cu^{2}
		\end{array}\right]$
	\end{center}
	Then,
	\begin{center}
		$\mathcal{M}(0,0)=\left[
		\begin{array}{cl}
		0 & 1 \\
		0 & 0
		\end{array}\right]$
	\end{center}
	We see that $\lambda^{2}=0$. According to the Theorem (\ref{5}), let $v+A(u,v)=0$ a solution of $v+A(u,v)=0$, where $A(u,v)=cu^{3}$ then $v=-cu^{3}$, also we have that $B(u,v)=cu^{2}v$, so $F(u)=-c^{2}u^{5}$ and $G(x)=4cu^{2}$ then $m=5$, $n=2$, $a=-c^{2}$, $b=4c$ and $m=2n+1$, furthermore $b^{2}+4a(n+1)\geq0$ . Therefore the origin of the system (\ref{2.9}) in infinite plane is an stable node if $c<0$ and unstable if $c>0$.
\end{proof}
\begin{center}
	For a more detailed study of the system see figure on the Poincaré sphere (\ref{fig:figura1}).
\end{center}

\subsection{Family II}
\textbf{In the Chart $U_{1}$} the associated system (\ref{familia2}):
\begin{equation}\rm
	\left\{
	\begin{array}{lcl}
		\dot{u}&=& -u^{2}v+2b  \\
		\dot{v}&=& -uv^{2}
	\end{array}\right.
\end{equation}
The system have not critical points at infinite plane.\\
\\
\textbf{In the Chart $U_{2}$} the associated system  (\ref{familia2}).
\begin{equation}\rm
	\left\{
	\begin{array}{lcl}
		\dot{u}&=& v-2bu^{2}  \\
		\dot{v}&=& -2buv \label{2.10}
	\end{array}\right.
\end{equation}
\begin{prop}\rm The point $(0,0)$ have one hyperbolic and one elliptic sector.
\end{prop}
\begin{proof}\rm The critical points associated with the system (\ref{2.10}) is $P:(0,0)$.\\
	Jacobian matrix is:
	\begin{center}
		$\mathcal{M}(u,v)=\left[
		\begin{array}{lc}
		-4bu & 1 \\
		-2bv & -2bu
		\end{array}\right]$
	\end{center}
	Then,
	\begin{center}
		$\mathcal{M}(0,0)=\left[
		\begin{array}{cl}
		0 & 1 \\
		0 & 0
		\end{array}\right]$
	\end{center}
	We see that, $\lambda^{2}=0$. According to the Theorem (\ref{5}), let $v+A(u,v)=0$  a solution of $v+A(u,v)=0$, where $A(u,v)=-2bu^{2}$ then $v=-2bu^{2}$, also we have that $B(u,v)=-2buv$, so $F(u)=-4b^{2}u^{3}$ and $G(x)=-6bu$ then $m=2n+1$ and $b^{2}+4a(n+1)$. Therefore the origin of the system (\ref{2.10}) in infinite plane have one hyperbolic and one elliptic sector.
\end{proof}
\begin{center}
	For a more detailed study of the system see figure on the Poincaré sphere (\ref{fig:figura2}).
\end{center}

\subsection{Family III}
\textbf{In the Chart $U_{1}$} the associated system (\ref{familia3}):
\begin{equation}\rm
	\left\{
	\begin{array}{lcl}
		\dot{u}&=& -u^{2}v+2a  \\
		\dot{v}&=& -uv^{2}
	\end{array}\right.
\end{equation}
The system have not critical points at infinite plane.\\
\\
\textbf{In the Chart $U_{2}$} the associated system (\ref{familia3}).
\begin{equation}\rm
	\left\{
	\begin{array}{lcl}
		\dot{u}&=& v-2au^{2}  \\
		\dot{v}&=& -2auv \label{2.11}
	\end{array}\right.
\end{equation}
\begin{prop}\rm The point $(0,0)$  have one hyperbolic and one elliptic sector.
\end{prop}
\begin{proof}\rm The critical points associated with the system (\ref{2.11}) is $P:(0,0)$.\\
	Jacobian matrix is:
	\begin{center}
		$\mathcal{M}(u,v)=\left[
		\begin{array}{lc}
		-4au & 1 \\
		-2av & -2au
		\end{array}\right]$
	\end{center}
	Then,
	\begin{center}
		$\mathcal{M}(0,0)=\left[
		\begin{array}{cl}
		0 & 1 \\
		0 & 0
		\end{array}\right]$
	\end{center}
	We see that, $\lambda^{2}=0$. According to the Theorem (\ref{5}), let $v+A(u,v)=0$  a solution of $v+A(u,v)=0$, where $A(u,v)=-2au^{2}$ then $v=-2au^{2}$, also we have that $B(u,v)=-2auv$, so $F(u)=-4a^{2}u^{3}$ y $G(x)=-6au$ then $m=2n+1$ and $b^{2}+4a(n+1)$. Therefore the origin of the system (\ref{2.11}) in infinite plane have one hyperbolic and one elliptic sector.
\end{proof}
\begin{center}
For a more detailed study of the system see figure on the Poincaré sphere (\ref{fig:figura3}).
\end{center}

\subsection{Family IV}
Let $d=a\left(p+4\right)$.\\
\\
\textbf{In the Chart $U_{1}$} the associated system (\ref{familia4}).
\begin{equation}\rm
	\left\{
	\begin{array}{cl}
		\dot{u}&= -u^{2}v+\frac{duv}{2}-\frac{3a^{2}v}{2}-c \\
		\dot{v}&= -uv^{2}
	\end{array}\right.
\end{equation}
The system have not critical points at infinite plane.\\
\\
\textbf{In the Chart $U_{2}$} the associated system (\ref{familia4}).
\begin{equation}\rm
	\left\{
	\begin{array}{lcl}
		\dot{u}&=& v-\frac{duv}{2}+\frac{3}{2}a^{2}u^{2}v+cu^{3}\\
		\dot{v}&=& -\frac{dv^{2}}{2}+\frac{3}{2}a^{2}uv^{2}+cu^{2}v
	\end{array}\right.\label{2.12}
\end{equation}
\begin{prop}\rm The point $(0,0)$ is an stable node if $c<0$ and unstable if $c>0$.
\end{prop}
\begin{proof} The critical points associated with the system (\ref{2.12}) is $P:(0,0)$.\\
	\\
	Jacobian matrix:
	\begin{center}
		$\mathcal{M}(u,v)=\left[
		\begin{array}{lc}
		-\frac{dv}{2}+3a^{2}v+3cu^{2} & 1-\frac{du}{2}+\frac{3}{2}a^{2}u^{2} \\
		\frac{3}{2}a^{2}v^{2} & -dv+3a^{2}uv+cu^{2}
		\end{array}\right]$
	\end{center}
	Then,
	\begin{center}
		$\mathcal{M}(0,0)=\left[
		\begin{array}{cl}
		0 & 1 \\
		0 & 0
		\end{array}\right]$
	\end{center}
	We see that, $\lambda^{2}=0$. According to the Theorem (\ref{5}), Let  $v=f(u)$ a solution of $v+A(u,v)=0$ where $v=f(u)=-cu^{3}+\ldots$ an approximation of the Taylor series solution, furthermore $B(u,v)=-\frac{dv^{2}}{2}+\frac{3}{2}a^{2}uv^{2}+cu^{2}v$, then $F(u)=-c^{2}u^{5}+\ldots$ and $G(u)=cu^{2}+\ldots$, so $m=5$,$n=2$,$b=4c$ and $a=-c^{2}$. Therefore the origin in the infinite plane is an stable node if $c<0$ and unstable if $c>0$..
\end{proof}
\begin{center}
	For a more detailed study of the system see figure on the Poincaré sphere (\ref{fig:figura4}).
\end{center}

\subsection{Family V}
Let $d=a\left( s+4 \right)$.\\
\\
\textbf{In the Chart $U_{1}$} the associated system (\ref{familia5}).
\begin{equation}\rm
	\left\{
	\begin{array}{cl}
		\dot{u}&= -u^{2}v+\frac{duv}{2}-\frac{3bv}{2}-c \\
		\dot{v}&= -uv^{2}
	\end{array}\right.
\end{equation}
The system have not critical points at infinite plane.\\
\\
\textbf{In the Chart $U_{2}$} the associated system  (\ref{familia5}
\begin{equation}\rm
	\left\{
	\begin{array}{lcl}
		\dot{u}&=& v-\frac{duv}{2}+\frac{3}{2}bu^{2}v+cu^{3}\\
		\dot{v}&=& -\frac{dv^{2}}{2}+\frac{3}{2}buv^{2}+cu^{2}v \label{2.13}
	\end{array}\right.
\end{equation}
\begin{prop}\rm The point $(0,0)$ is stable node if $c<0$ and unstable if $c>0$.
\end{prop}
\begin{proof} The critical points associated with the system (\ref{2.13}) is $P:(0,0)$.\\
	\\
	Jacobian matrix:
	\begin{center}
		$\mathcal{M}(u,v)=\left[
		\begin{array}{lc}
		-\frac{dv}{2}+3buv+3cu^{2} & 1-\frac{du}{2}+\frac{3}{2}bu^{2} \\
		\frac{3}{2}bv^{2}+2cuv & -2dv+3buv+cu^{2}
		\end{array}\right]$
	\end{center}
	Then,
	\begin{center}
		$\mathcal{M}(0,0)=\left[
		\begin{array}{cl}
		0 & 1 \\
		0 & 0
		\end{array}\right]$
	\end{center}
	We see that, $\lambda^{2}=0$. According to the Theorem (\ref{5}), let $v=f(u)$ a solution of $v+A(u,v)=0$, where $v=f(u)=-cu^{3}+\ldots$ approximation of the Taylor series solution, furthermore $B(u,v)=-\frac{dv^{2}}{2}+\frac{3}{2}buv^{2}+cu^{2}v$, then $F(u)=-c^{2}u^{5}+\ldots$ y $G(u)=cu^{2}+\ldots$, so $m=5$,$n=2$,$b=4c$ y $a=-c^{2}$. Therefore the origin in the infinite plane is an stable node if $c<0$ and is a unstable node if $c>0$.
\end{proof}

For a more detailed study of the system which can see figure on the Poincarè sphere (\ref{fig:figura5}) and (\ref{fig:figura6}) .

\section{Global Phase Portrait}
In this section We show the global phase portrait associate to each family:

\begin{figure}[htb!]
	\centering
		\includegraphics[width=50mm]{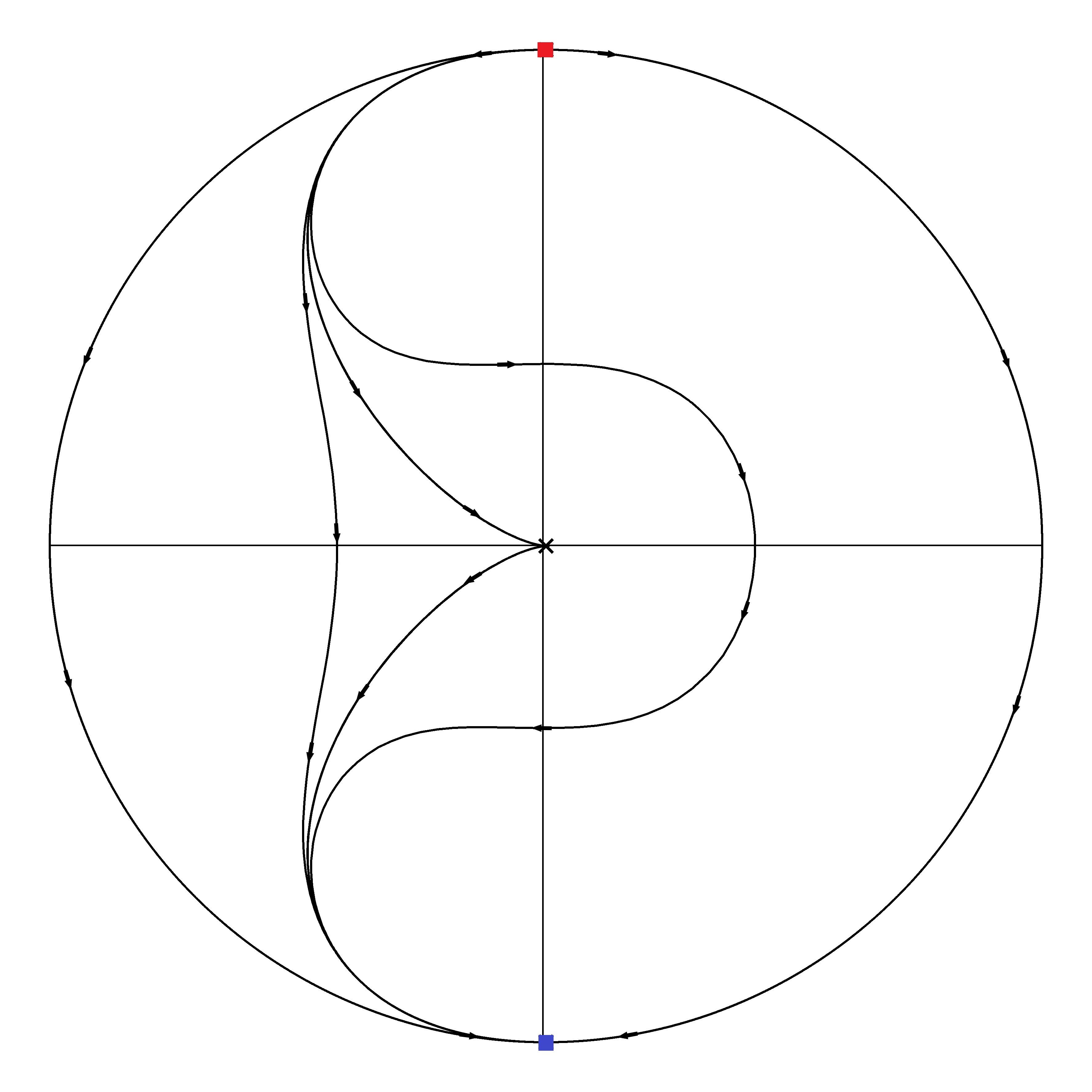}
		\caption{Family I}
		\label{fig:ifigura1}
\end{figure}

\begin{figure}[htb]
	\centering

		\includegraphics[width=40mm]{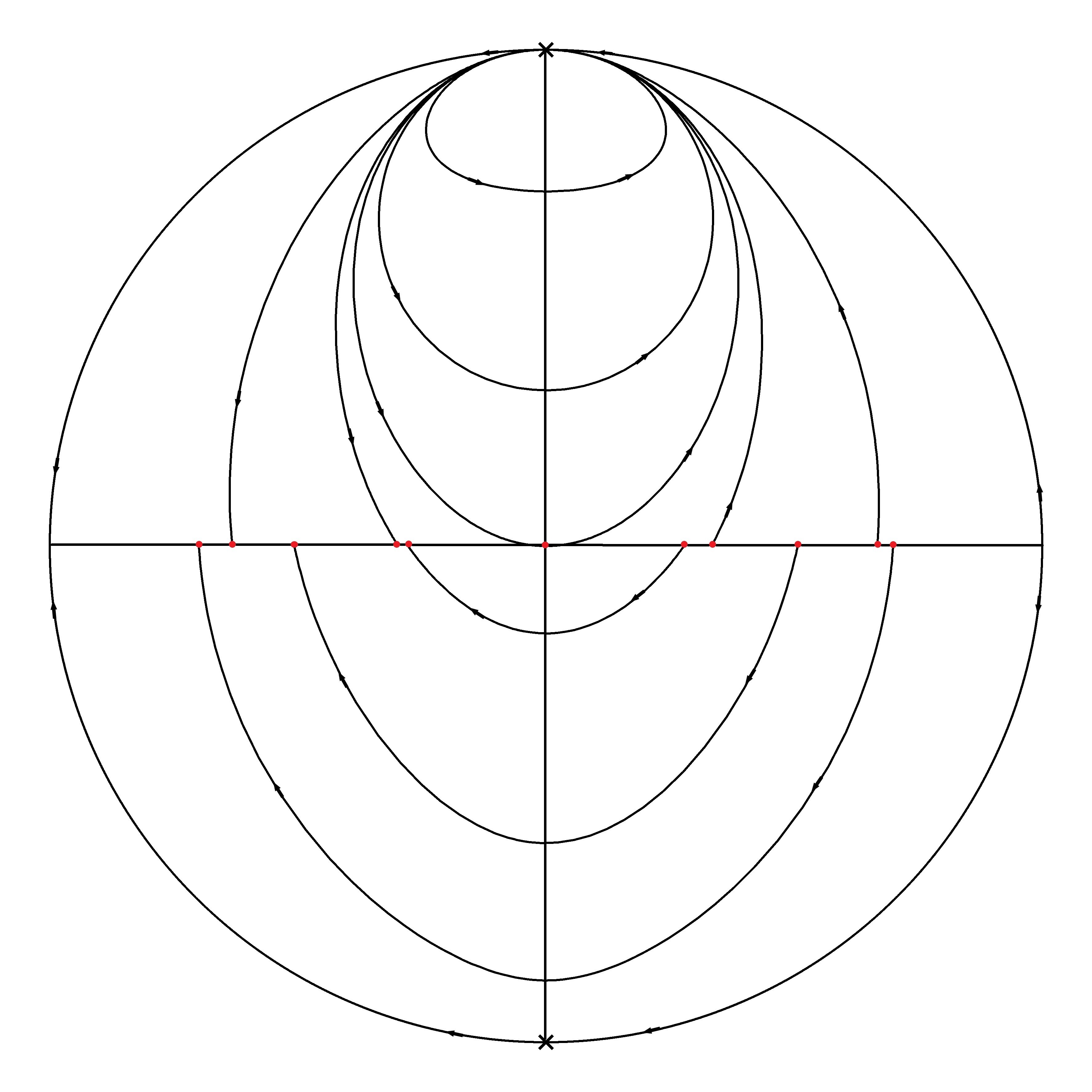}
		\caption{Family II}
		\label{fig:ifigura2}
\end{figure}

\begin{figure}[htb]	
	\centering

		\includegraphics[width=50mm]{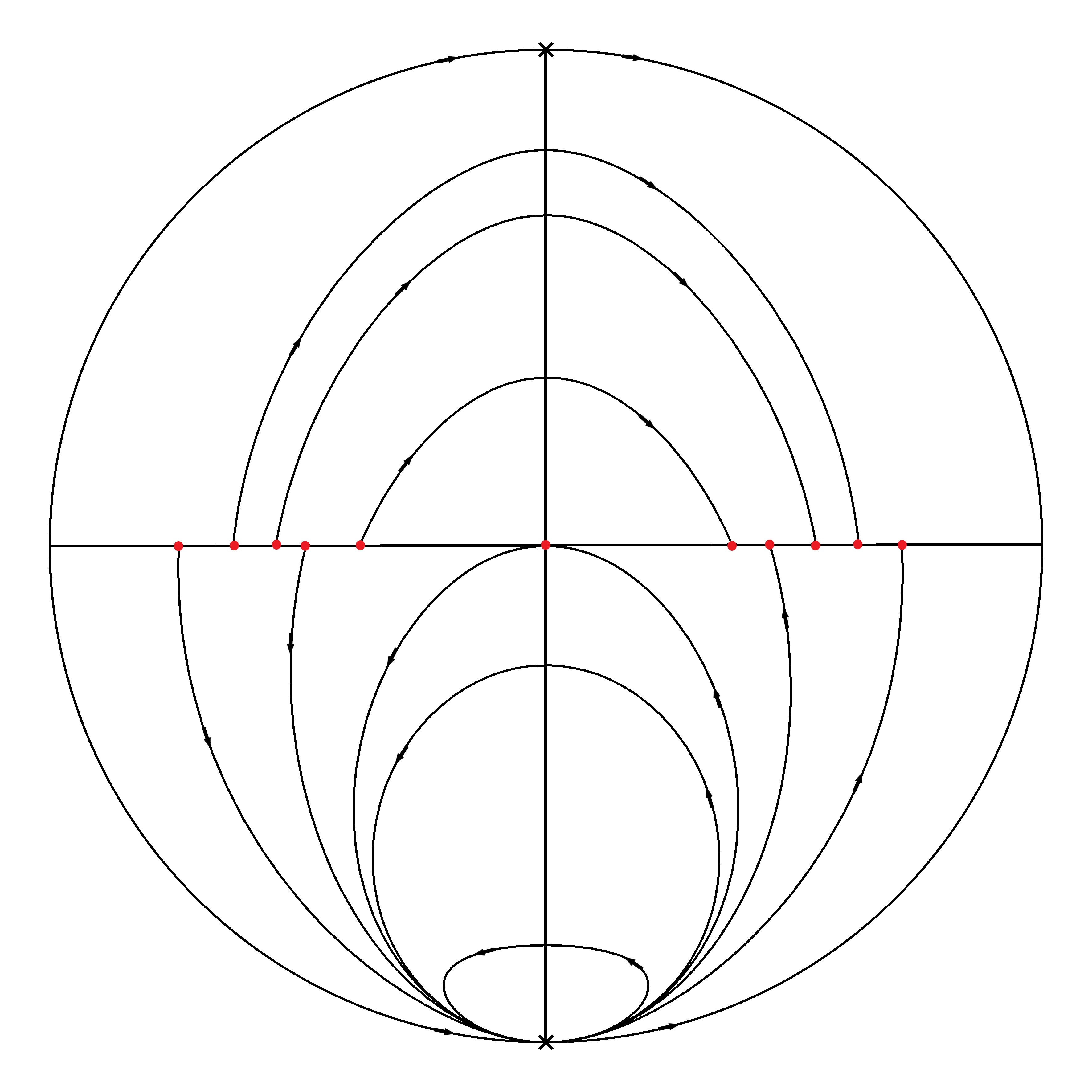}
		\caption{Family III}
		\label{fig:ifigura3}
\end{figure}

\begin{figure}[htb]
	\centering

		\includegraphics[width=50mm]{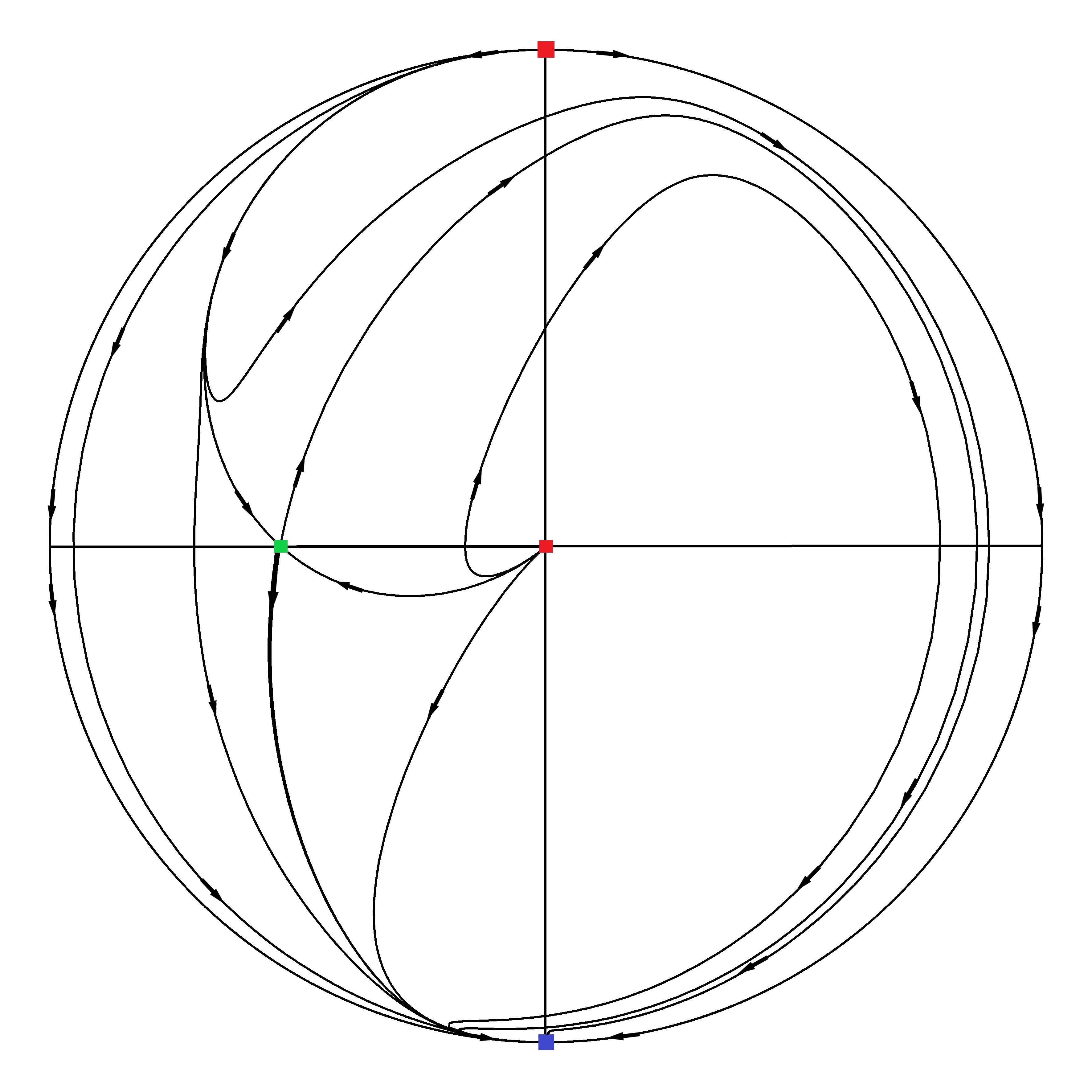}
		\caption{Family IV}
		\label{fig:ifigura4}

	\centering
		\includegraphics[width=50mm]{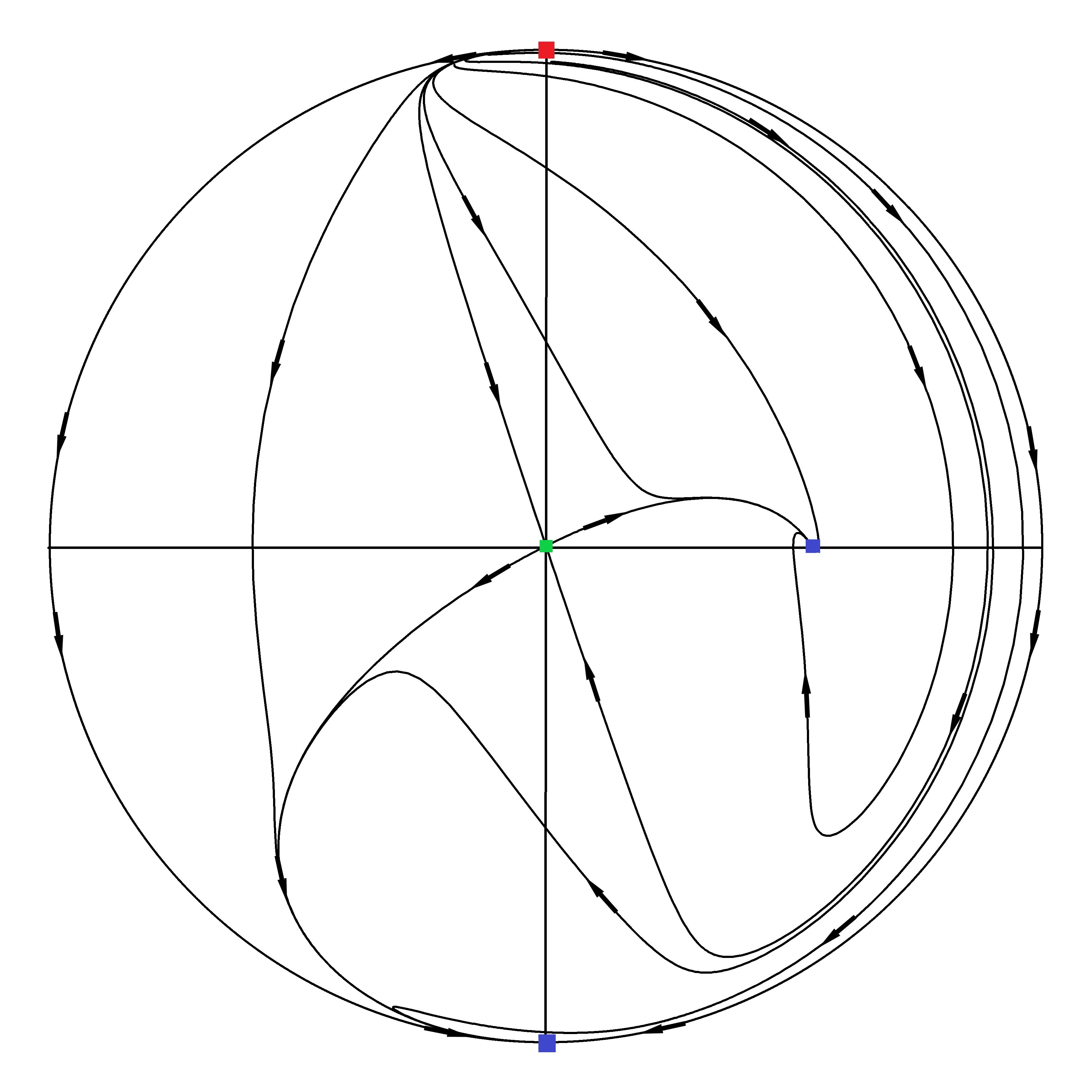}
		\caption{Family V when $b<0$.}
		\label{fig:ifigura5}
\end{figure}

\begin{figure}[htb!]
	\centering

		\includegraphics[width=50mm]{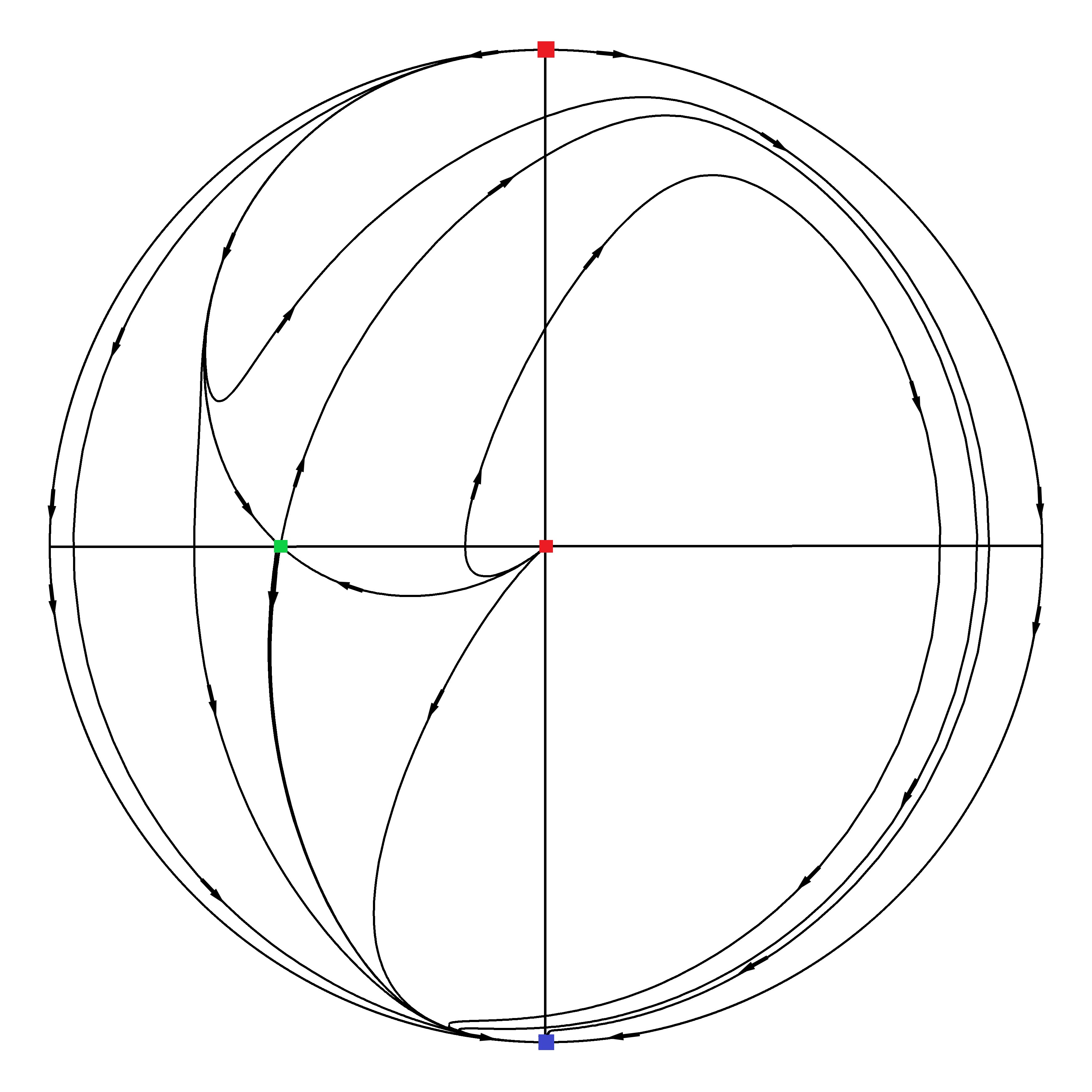}
		\caption{Family V when $b>0$.}
		\label{fig:ifigura6}

\end{figure}

\newpage
\vspace{1cm}
\section{Algebraic Aspects}
In this section we analyze the families I, II, III, IV and V through an algebraic point of view. We compute the solutions in terms of P-Weierstrass function of families I, II ($p=-4$) and V ($s=-4$), as well the differential Galois group of their variational equations.
\subsection{Family I}
\begin{teor}\label{thprim1}
Consider the family I, The following statements hold.
\begin{enumerate}
\item The dynamical system is hamiltonian with one degree of freedom and with polynomial first integral $$H=H(x,y)=\frac{y^2}{2}+\frac{c}{3}x^3$$
\item The integral curve of the Hamiltomian vector field is $$(-\frac{6}{c}\wp(t+k_0;0,-2H),-\frac{6}{c}\dot\wp(t+k_0;0,-2H)).$$
\item The Differential Galois Group associated to the foliation is isomorphic to $\mathbb{Z}_2$.
\item The connected identity component of the Differential Galois Group of the first variational equation along any particular solution is an abelian group.
\end{enumerate}
\end{teor}
\begin{proof}
    We proceed according to each item,
    \begin{enumerate}
        \item The polynomial vector field related with family I is equivalent to Equation \eqref{lmp1} being $f(x)=-cx^2$. In virtue of Equation \eqref{lmp2} we have the Hamiltonian $H=\frac{y^2}{2}+\frac{c}{3}x^3$.
        \item Due to $y=\dot x$, we obtain $y^2=-\frac{2c}{3}^2+H$. Through the change of variable $(x,y)\mapsto(\sqrt[3]{\frac{-6}{c}}x,\sqrt[3]{\frac{-6}{c}}y)$, we arrive to the elliptic curve given in Equation \eqref{pwei0} with invariants $g_2=0$ and $g_3=-2H$. Thus, the integral curve of the Hamiltonian system is $(x,\dot x)$, being $x$ given by $-\frac{6}{c}\wp(t+k_0;0,-2H)$.
        \item The foliation associated to the vector field of Family I is $$y'=-\frac{cx^2}{y},\quad ':=\frac{d}{dx}.$$ Setting $z=\frac{y^2}{2}$, we obtain $z'=-cx^2$ and therefore $z=-\frac{c}{3}x^3$. Due to the differential field $K$ is the field of rational functions $\mathbb{C}(x)$, $\sigma(z)=z$ and $\sigma(y)=\lambda \sqrt{z}$, where $\lambda^2=1$. Thus, the Picard-Vessiot extension $L$ is a quadratic extension of $K$ and we can conclude that $DGal(L/K)$ has two elements.
        \item Let $(x_0(t),\dot x_0(t))$ be a particular solution of the polynomial vector field related with Family I. Thus, the first variational equation  is $$\frac{d}{dt}\begin{pmatrix}\xi_1\\\xi_2
        \end{pmatrix}=\begin{pmatrix}0&1\\-2cx_0(t)&0\end{pmatrix}\begin{pmatrix}\xi_1\\\xi_2
        \end{pmatrix},$$ which is equivalent to $\ddot \xi=-2cx_0(t)\xi$, being $\xi=\xi_1$. By Morales-Ramis theory, due to the dynamical system is polynomially integrable, the differential Galois group of this first variational equation is abelian.
    \end{enumerate}
\end{proof}

\subsection{Family II}
\begin{teor}\label{thprim2}
Consider the family II, The following statements hold.
\begin{enumerate}
\item The first integral of the polynomial vector field is $$I=I(x,y)=y-bx^2$$
\item The integral curve of the polynomial vector field is $(x(t),\dot x(t))$, where  $$x(t)=\sqrt{\frac{k_1}{b}}\tan(\sqrt{k_1b(k_2+t)}).$$
\item The Differential Galois Group associated to the foliation is isomorphic to the identity group.
\item The connected identity component of the Differential Galois Group of the first variational equation around any particular solution is an abelian group.
\end{enumerate}
\end{teor}
\begin{proof}
    We proceed according to each item,
    \begin{enumerate}
        \item The total derivative of $I(x,y)$ vanishes, i.e., $\dot I=0$, therefore $I$ is a first integral of the vector field related to family II.
        \item Due to $y=\dot x$, we obtain $\ddot x=b\dot z$, where $z=x^2$. Thus, $\dot{x}=bx^2+k_1$, which implies that $$\int\frac{dx}{bx^2+k_1}=t+k_2$$ and then $x(t)=\sqrt{\frac{k_1}{b}}\tan(\sqrt{k_1b(k_2+t)}).$.
        \item The foliation associated to the vector field of Family II is $$y'=2bx,\quad ':=\frac{d}{dx}.$$ Then the solution of this foliation is $$y(x)=bx^2+k_1$$. Then we can conclude that $DGal(L/K)$ has one element, i.e., $DGal(L/K)= I_2$.
        \item Let $(x_0(t),\dot x_0(t))$ be a particular solution of the polynomial vector field related with Family II. Thus, the first variational equation  is $$\frac{d}{dt}\begin{pmatrix}\xi_1\\\xi_2
        \end{pmatrix}=\begin{pmatrix}0&1\\-2by_0(t)&2bx_0(t)\end{pmatrix}\begin{pmatrix}\xi_1\\\xi_2
        \end{pmatrix},$$ which is equivalent to $$\ddot \xi-2bx_0(t)\dot\xi-2by_0(t)\xi=0,\,\xi=\xi_1.$$ Due to the first integral is of polynomial type, by Morales-Ramis theory we can conclude that the connected identity component of the differential Galois group of the first variational equation along any particular solution is an abelian group.
    \end{enumerate}
\end{proof}
\subsection{Family III}
\begin{teor}\label{thprim3}
Consider the family III, The following statements hold.
\begin{enumerate}
\item The first integral of the polynomial vector field is $$I=I(x,y)=y-ax^2$$
\item The integral curve of the polynomial vector field is $(x(t),\dot x(t))$, where  $$x(t)=\sqrt{\frac{k_1}{a}}\tan(\sqrt{k_1a(k_2+t)}).$$
\item The Differential Galois Group associated to the foliation is isomorphic to the identity group.
\item The connected identity component of the Differential Galois Group of the first variational equation around any particular solution is an abelian group.
\end{enumerate}
\end{teor}
\begin{proof}
    We proceed according to each item,
    \begin{enumerate}
        \item The total derivative of $I(x,y)$ vanishes, i.e., $\dot I=0$, therefore $I$ is a first integral of the vector field related to family III.
        \item Due to $y=\dot x$, we obtain $\ddot x=a\dot z$, where $z=x^2$. Thus, $\dot{x}=ax^2+k_1$, which implies that $$\int\frac{dx}{ax^2+k_1}=t+k_2$$ and then $x(t)=\sqrt{\frac{k_1}{a}}\tan(\sqrt{ak_1(k_2+t)}).$.
        \item The foliation associated to the vector field of Family II is $$y'=2ax,\quad ':=\frac{d}{dx}.$$ Then the solution of this foliation is $$y(x)=ax^2+k_1$$. Then we can conclude that $DGal(L/K)$ has one element, i.e., $DGal(L/K)= I_2$.
        \item Let $(x_0(t),\dot x_0(t))$ be a particular solution of the polynomial vector field related with Family III. Thus, the first variational equation  is $$\frac{d}{dt}\begin{pmatrix}\xi_1\\\xi_2
        \end{pmatrix}=\begin{pmatrix}0&1\\-2ay_0(t)&2ax_0(t)\end{pmatrix}\begin{pmatrix}\xi_1\\\xi_2
        \end{pmatrix},$$ which is equivalent to $$\ddot \xi-2ax_0(t)\dot\xi-2ay_0(t)\xi=0,\,\xi=\xi_1.$$ Due to the first integral is of polynomial type, by Morales-Ramis theory we can conclude that the connected identity component of the differential Galois group of the first variational equation along any particular solution is an abelian group.
    \end{enumerate}
\end{proof}
\subsection{Family IV}
\begin{teor}\label{thprim4}
Consider the family IV, being $p=-4$. The following statements hold.
\begin{enumerate}
\item The dynamical system is hamiltonian with one degree of freedom and with polynomial first integral $$H=H(x,y)=\frac{y^2}{2}+\frac{c}{3}x^3+\frac{3}{4}a^2x^2$$
\item The integral curve of the Hamiltomian vector field is given in terms of P-function. 
\item The Differential Galois Group associated to the foliation is isomorphic to $\mathbb{Z}_2$.
\item The connected identity component of the Differential Galois Group of the first variational equation along any particular solution is an abelian group.
\end{enumerate}
\end{teor}
\begin{proof}
    We proceed according to each item,
    \begin{enumerate}
        \item The polynomial vector field related with family IV is equivalent to Equation \eqref{lmp1} being $f(x)=-cx^2-\frac{3}{2}a^2x$. In virtue of Equation \eqref{lmp2} we have the Hamiltonian $H=\frac{y^2}{2}+\frac{c}{3}x^3+\frac{3}{4}a^2x^2$.
        \item Due to $y=\dot x$, we obtain $y^2=-\frac{2c}{3}^2-\frac{3}{2}a^2x+2H$. Because previous expression is a cubic polynomial in $x$, we can do a suitable change of variable to arrive to the elliptic curve given in Equation \eqref{pwei0} with invariants $g_2$ and $g_3$. Thus, the integral curve of the Hamiltonian system is written in terms of P-function. 
        \item The foliation associated to the vector field of Family IV is $$y'=-\frac{cx^2-\frac{3}{2}a^2x}{y},\quad ':=\frac{d}{dx}.$$ Setting $z=\frac{y^2}{2}$, we obtain $z'=-cx^2-\frac{3}{2}a^2x$ and therefore $z=-\frac{c}{3}x^3-\frac{3}{4}a^2x^2$. Due to the differential field $K$ is the field of rational functions $\mathbb{C}(x)$, $\sigma(z)=z$ and $\sigma(y)=\lambda \sqrt{z}$, where $\lambda^2=1$. Thus, the Picard-Vessiot extension $L$ is a quadratic extension of $K$ and we can conclude that $DGal(L/K)$ has two elements.
        \item Let $(x_0(t),\dot x_0(t))$ be a particular solution of the polynomial vector field related with Family IV. Thus, the first variational equation  is $$\frac{d}{dt}\begin{pmatrix}\xi_1\\\xi_2
        \end{pmatrix}=\begin{pmatrix}0&1\\-\frac{3}{2}a^2-2cx_0(t)&0\end{pmatrix}\begin{pmatrix}\xi_1\\\xi_2
        \end{pmatrix},$$ which is equivalent to $\ddot \xi=(-\frac{3}{2}a^2-2cx_0(t))\xi$, being $\xi=\xi_1$. By Morales-Ramis theory, due to the dynamical system is polynomially integrable, the differential Galois group of this first variational equation is abelian.
    \end{enumerate}
\end{proof}
\subsection{Family V}
\begin{teor}\label{thprim5}
Consider the family V, being $s=-4$. The following statements hold.
\begin{enumerate}
\item The dynamical system is hamiltonian with one degree of freedom and with polynomial first integral $$H=H(x,y)=\frac{y^2}{2}+\frac{c}{3}x^3+\frac{3}{4}bx^2$$
\item The integral curve of the Hamiltomian vector field is given in terms of P-function. 
\item The Differential Galois Group associated to the foliation is isomorphic to $\mathbb{Z}_2$.
\item The connected identity component of the Differential Galois Group of the first variational equation along any particular solution is an abelian group.
\end{enumerate}
\end{teor}
\begin{proof}
    We proceed according to each item,
    \begin{enumerate}
        \item The polynomial vector field related with family V is equivalent to Equation \eqref{lmp1} being $f(x)=-cx^2-\frac{3}{2}bx$. In virtue of Equation \eqref{lmp2} we have the Hamiltonian $H=\frac{y^2}{2}+\frac{c}{3}x^3+\frac{3}{4}bx^2$.
        \item Due to $y=\dot x$, we obtain $y^2=-\frac{2c}{3}^2-\frac{3}{2}bx+2H$. Because previous expression is a cubic polynomial in $x$, we can do a suitable change of variable to arrive to the elliptic curve given in Equation \eqref{pwei0} with invariants $g_2$ and $g_3$. Thus, the integral curve of the Hamiltonian system is written in terms of P-function. 
        \item The foliation associated to the vector field of Family V is $$y'=-\frac{cx^2-\frac{3}{2}bx}{y},\quad ':=\frac{d}{dx}.$$ Setting $z=\frac{y^2}{2}$, we obtain $z'=-cx^2-\frac{3}{2}bx$ and therefore $z=-\frac{c}{3}x^3-\frac{3}{4}bx^2$. Due to the differential field $K$ is the field of rational functions $\mathbb{C}(x)$, $\sigma(z)=z$ and $\sigma(y)=\lambda \sqrt{z}$, where $\lambda^2=1$. Thus, the Picard-Vessiot extension $L$ is a quadratic extension of $K$ and we can conclude that $DGal(L/K)$ has two elements.
        \item Let $(x_0(t),\dot x_0(t))$ be a particular solution of the polynomial vector field related with Family V. Thus, the first variational equation  is $$\frac{d}{dt}\begin{pmatrix}\xi_1\\\xi_2
        \end{pmatrix}=\begin{pmatrix}0&1\\-\frac{3}{2}b-2cx_0(t)&0\end{pmatrix}\begin{pmatrix}\xi_1\\\xi_2
        \end{pmatrix},$$ which is equivalent to $\ddot \xi=(-\frac{3}{2}b-2cx_0(t))\xi$, being $\xi=\xi_1$. By Morales-Ramis theory, due to the dynamical system is polynomially integrable, the differential Galois group of this first variational equation is abelian.
    \end{enumerate}
\end{proof}
\section{Conclusion}
An in-depth analysis of the quadratic systems containing certain multiparametric families was carried out,   for this purpose they were identified and classified, with the aim of making more bearable the study on the stability of its critical points both in the finite and infinite plane. The existence of transcritical bifurcations in the given system was determined. 
Finally, a study was made on the hamiltonian cases and the differential Galois groups of their foliations and variational equations.

\end{document}